\documentclass[10pt]{amsart}
\usepackage{calrsfs}
\usepackage{graphicx}
\usepackage{mathabx}
\usepackage{hyperref}
\usepackage{amsthm}
\usepackage{amssymb}
\usepackage{multirow}
\usepackage{tikz-cd}
\usepackage{amsmath}
\usepackage{todonotes}
\usepackage{amsbsy}
\usepackage[all]{xy}
\usepackage{changes}
\usepackage{enumerate}
\usepackage{comment}
\usepackage{mathtools}
\usepackage{changes}
\usepackage{mdframed}

\usepackage{lscape}
\usepackage{tabularx}
\usepackage{booktabs}
\usepackage{stmaryrd}
\usepackage[labelfont=bf,format=plain,justification=raggedright,singlelinecheck=false]{caption}
\usepackage[ruled,vlined,resetcount]{algorithm2e}
\usepackage{mathrsfs}
\usepackage{empheq}
\definechangesauthor[color=orange,name={Aristides Kontogeorgis}]{AK}
 \definechangesauthor[color=red, name=Kostas Karagiannis]{KK}

\usepackage[margin=1 in]{geometry}
\usetikzlibrary{patterns,arrows,decorations.pathreplacing}

\newtheorem{theorem}{Theorem}[subsection]
\newtheorem*{theorem*}{Theorem}
\newtheorem{lemma}[theorem]{Lemma}
\newtheorem{corollary}[theorem]{Corollary}
\newtheorem{proposition}[theorem]{Proposition}
\theoremstyle{definition}

\newtheorem{remark}[theorem]{Remark}
\newtheorem{definition}[theorem]{Definition}

\def\p{6} 

\newcount\rows
\rows=\numexpr\p-1\relax

\newcommand{\cO}{\mathcal{O}}

\DeclareMathOperator{\Sym}{Sym}



\date{\today}

\title{Equivariant Koszul Cohomology of Canonical Curves}

  \author[K. Karagiannis]{Kostas Karagiannis}
  \address{Department of Mathematics, National and Kapodistrian  University of Athens
  Pane\-pist\-imioupolis, 15784 Athens, Greece}
  \email{konstantinos.v.karagiannis@gmail.com}

  \author[A. Kontogeorgis]{Aristides Kontogeorgis}
  \address{Department of Mathematics, National and Kapodistrian  University of Athens
  Pane\-pist\-imioupolis, 15784 Athens, Greece}
  \email{kontogar@math.uoa.gr}

  \author[K. Manousou Sotiropoulou]{Konstantia Manousou Sotiropoulou}
  \address{Department of Mathematics, National and Kapodistrian  University of Athens
  Pane\-pist\-imioupolis, 15784 Athens, Greece}
  \email{komanous@math.uoa.gr}

\date \today

\makeatletter
\newcommand{\aprod}{\mathop{\operator@font \hbox{\Large$\ast$}}}
\makeatother

\begin{document}

\begin{abstract}
This paper investigates the representation-theoretic structure of the Koszul cohomology of a smooth projective variety $X$ over an algebraically closed field $k$, admitting an action of a finite group $G$ of order coprime to ${\rm char}(k)$. Properties of $G$-equivariant functors are employed to show that the associated Koszul complex is a complex of $kG$-modules, and to generalize known dimension formulas to identities between virtual representations. In the case of canonical curves, explicit formulas are obtained by combining the theory of equivariant Euler characteristics and equivariant Riemann-Roch theorems with that of generating functions for Schur functors.
\end{abstract}

\maketitle

\section{Introduction}

\subsection{Motivation and historical background}
The genealogy of the Koszul complex begins with Hilbert’s foundational 1890 paper \cite{MR1510634}, where it appears implicitly in his proof of the famous syzygy theorem that bears his name. However, its formal algebraic structure was not codified until Jean-Louis Koszul’s 1950 paper \cite{MR42428}. Though his primary motivation was topological, Koszul’s treatment provided the abstract definitions and results that characterize the modern understanding of the complex today.

The modern era of the theory began with Mark Green, who introduced the geometric counterpart of the Koszul complex \cite{MR739785} and demonstrated that geometric properties of projective varieties could be reinterpreted through homological vanishing theorems. Building on this foundation, Green and Lazarsfeld extended several classical results concerning syzygies of projective varieties into the language of Koszul cohomology, see the surveys \cite{MR1082354} and \cite{MR1082360}. Subsequent breakthroughs eventually resolved influential conjectures on the subject, leading to a period of significant growth for the field; see the preface of \cite{MR2573635} for a detailed account of this expansion and its various applications.  { Finite-group actions on graded modules and their minimal free resolutions have an established literature. In particular, Galetto~\cite{Galetto2022FiniteGroupCharacters} gives an algorithm computing the characters of a finite group on each term of a minimal free resolution, implemented in the Macaulay2 package \texttt{BettiCharacters}~\cite{Galetto2023BettiCharacters}. That approach works with the matrices of an explicit minimal free resolution (whose computation, as far as we know, requires a Gr\"obner basis computation in $g$ variables) and returns every individual character $[K_{p,q}]$. Our approach is complementary: for canonical curves it yields \emph{closed formulas} for the virtual differences $[K_{p,1}]-[K_{p-1,2}]$ in terms of the ramification data of $X\to X/G$ alone, with no resolution computed; combined with vanishing theorems (see \S\ref{sec:Kummer} and Remark~\ref{rem:individual}) {and under the vanishing predicted by Green's conjecture,} these determine individual characters outside a middle range.}

The paradigm of studying varieties through associated vector spaces naturally evolves into the study of group representations when a finite group acts on the variety via automorphisms. In this setting, the $G$-module structure of these spaces provides the key to understanding the underlying action. This is most classically illustrated by smooth projective curves, where the dictionary between geometry and function fields allows one to apply the machinery of Galois theory. By utilizing global-to-local methods, such as ramification theory and the study of decomposition groups, one can recover global geometric properties from the behavior of these field extensions.

Building on this perspective, the main aim of this paper is to investigate the behavior of the Koszul complex and its cohomology in the equivariant setting, where a finite group $G$ acts on a variety $X$. Our focus lies primarily on the case of smooth projective curves, with particular emphasis on canonical curves. We seek to explore the interplay between the geometry of the group action and the representation-theoretic structure of the Koszul cohomology groups, with the goal of obtaining results in both directions.

\subsection{Outline of the paper}
In Section~\ref{sec:equiv}, we consider a smooth irreducible scheme $X$ over an algebraically closed field $k$ acted upon by a finite group $G$ whose order is coprime to the characteristic of $k$. After reviewing the definitions of $G$-equivariant quasicoherent sheaves and $G$-equivariant functors, we establish in Proposition \ref{prop:G-functors} a general criterion for a functor to be $G$-equivariant. This allows us to verify in Proposition \ref{prop:exGfunctors} and Proposition \ref{prop:contrprop} that the constituent operations in Green and Lazarsfeld’s functorial construction of the Koszul complex are $G$-equivariant, while endowing them with canonical $G$-structures derived solely from the action on $X$. This framework culminates in Theorem \ref{th:main}, in which we not only show that the Koszul complex is a complex of $kG$-modules, but also express the equivariant structure of its cohomology in terms of Zariski cohomology, lifting the classical construction of Green and Lazarsfeld into the equivariant setting.

In Section~\ref{sec:canonical}, we focus on canonical curves and provide a formula for differences of Koszul representations in Corollary \ref{cor:bettivanish}, motivated by the non-equivariant setting where such differences are explicitly computable via Riemann–Roch. Within the equivariant framework, this formula naturally splits into two distinct components. The first directly relates to the theory of equivariant Euler characteristics; to evaluate it, we use the equivariant Riemann–Roch theorem of Ellingsrud and Lønsted (Theorem \ref{th:EL}) and provide in Theorem \ref{th:mainwedge} an explicit decomposition into a direct sum of irreducible representations. The second component remarkably turns out to be related to the theory of Schur functors, which are central to the representation theory of $\text{GL}_n$ and are intimately connected to the combinatorics of Young tableaux. {To organize the irreducible decompositions of these modules uniformly,} we provide explicit generating functions both for the character values of this module in Theorem \ref{th:genchars} and for its class in the representation ring of $G$ in Theorem \ref{th:Molien}. Finally, in Section \ref{sec:Kummer}, we apply the above mentioned formulas to get explicit values in the case of trigonal Kummer curves and demonstrate the validity of our results by retrieving them using an alternative, independent method.
\bigskip

\noindent {\bf Acknowledgements} The authors would like to thank Miltiadis Karakikes  for helpful comments on early versions of this paper. The research project is implemented in the framework of H.F.R.I Call “Basic research Financing (Horizontal support of all Sciences)” under the National Recovery and Resilience Plan “Greece 2.0” funded by the European Union Next Generation EU (H.F.R.I.
Project Number: 14907).
\begin{center}
\includegraphics[scale=0.4]{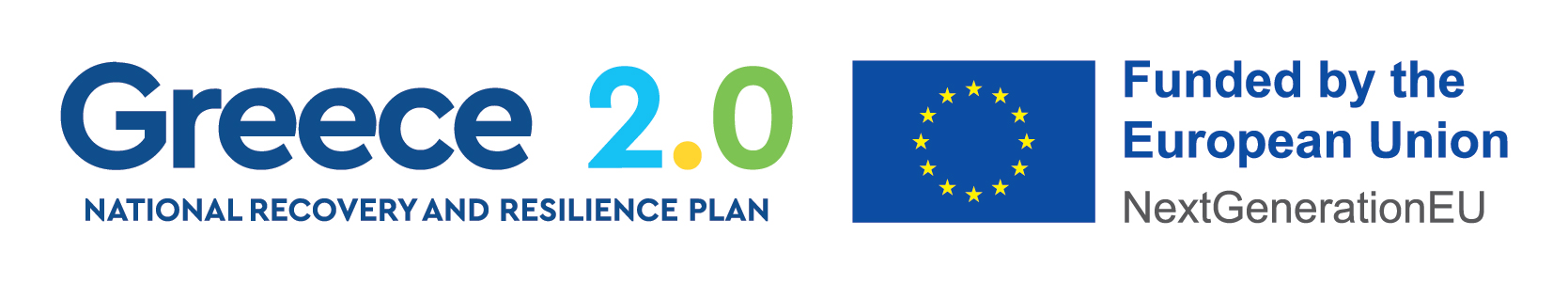}
\hskip 1cm
\includegraphics[scale=0.05]{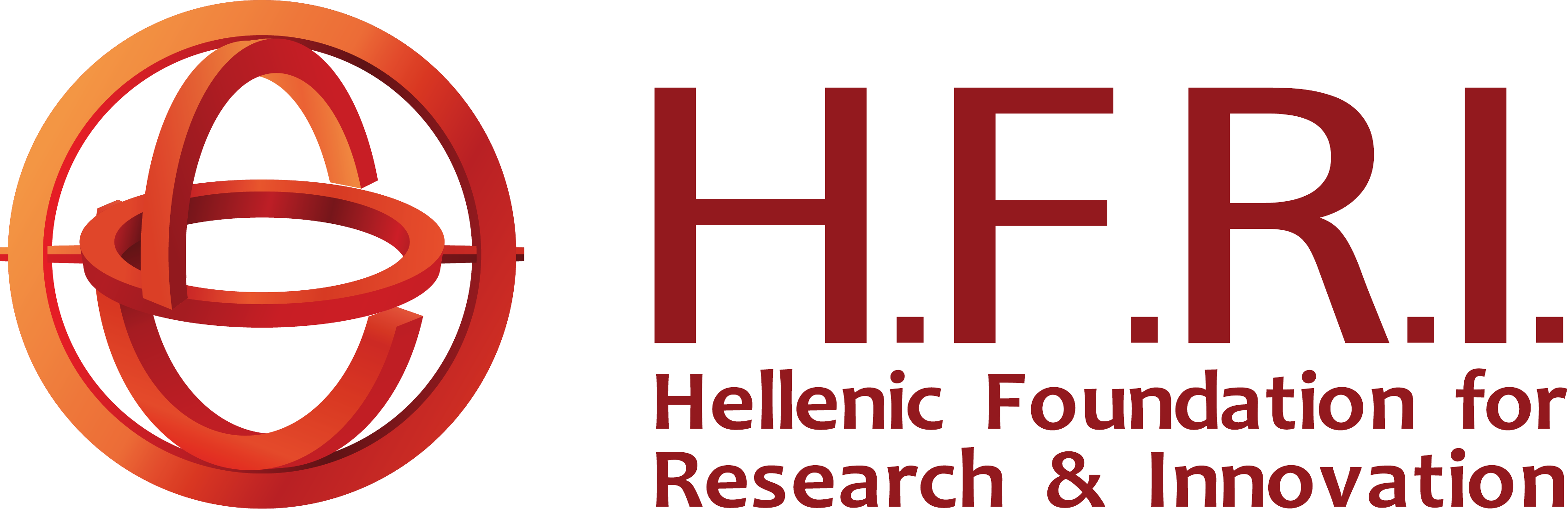}
\end{center}

\section{The Equivariant Koszul Complex}\label{sec:equiv}
Let $(X,\mathcal{O}_X)$ be a smooth, projective, irreducible scheme over an algebraically closed field $k$, and let $G$ be a finite group of order coprime to the characteristic of $k$ acting faithfully on $X$. {Writing $\rho_g:X\rightarrow X$ for the automorphism corresponding to $g\in G$, we use the convention $\rho_{gh}=\rho_g\circ\rho_h$, so that $\rho_{gh}^*\cong\rho_h^*\rho_g^*$ canonically; canonical identifications between iterated pullbacks are understood in all formulas below. With this convention the induced actions on spaces of global sections satisfy $\theta_{gh}=\theta_h\theta_g$, i.e. the $kG$-modules arising in this way are right $kG$-modules (for abelian $G$, as in Section~\ref{sec:Kummer}, this distinction disappears).} This section reviews the machinery necessary to define the Koszul complex and its cohomology in the $G$-equivariant setting. Many of the results leading up to the main theorem of this section are not original; the goal is to make the treatment as self-contained as possible.

\subsection{Equivariant quasicoherent sheaves}
Let $\rho_g:X\rightarrow X$ be the automorphism of $X$ corresponding to $g\in G${, with $\rho_{gh}=\rho_g\circ\rho_h$}. Pulling back the induced isomorphism of sheaves $\rho^\#_g:\mathcal{O}_X\rightarrow (\rho_g)_*\mathcal{O}_X$ along $\rho_g$ yields a canonical isomorphism
\[
c_g:\rho^*_g\mathcal{O}_X=\rho_g^{-1}\mathcal{O}_X\otimes_{\rho_g^{-1}\mathcal{O}_X} \mathcal{O}_X\overset{\rho^\#_g}{\longrightarrow}\rho_g^{-1}((\rho_g)_*\mathcal{O}_X)=\mathcal{O}_X,
\]
which satisfies $c_{gh}=c_h\circ\rho^*_h(c_g)$, for all $g,h\in G$. The collection $\{c_g\}_{g\in G}$ will be referred to as {\em the $G$-linearization of $\mathcal{O}_X$}.

\begin{definition}[{\cite[Ch.~1, \S 3, Def.~1.6]{MR1304906}}]\label{def:G-sheaf}
The category ${\rm QCoh}_G(X)$ of $G$-equivariant quasicoherent sheaves on $X$ is defined as follows.

\begin{enumerate}
\item The objects of ${\rm QCoh}_G(X)$ are pairs $(\mathcal{F},\{\phi_g\}_{g\in G})$, where $\mathcal{F}$ is a quasicoherent sheaf on $X$ and  $\phi_g:\rho^*_g\mathcal{F}\rightarrow\mathcal{F}$ is an isomorphism satisfying $\phi_{gh}=\phi_h\circ\rho^*_h(\phi_g)$ for all $g,h\in G$.
\[
\begin{tikzcd}
\rho_{gh}^*\mathcal F \arrow[rr, "\phi_{gh}"] \arrow[dr,"\rho_h^*(\phi_g)"'] & &
\mathcal F  \\
& \rho_h^*\mathcal F \arrow[ur,"\phi_h"'] &
\end{tikzcd}
\]
\item The arrows of ${\rm QCoh}_G(X)$ between objects $(\mathcal{F},\{\phi_g\}_{g\in G})$ and $(\mathcal{G},\{\psi_g\}_{g\in G})$ are morphisms $f:\mathcal{F}\rightarrow\mathcal{G}$ of quasicoherent sheaves satisfying $\psi_g\circ \rho^*_g(f) =f\circ\phi_g$ for all $g\in G$.
\[
\begin{tikzcd}
\rho_g^*\mathcal F \arrow[r,"\rho_g^*(f)"] \arrow[d,"\varphi_g"'] & \rho_g^*\mathcal G \arrow[d,"\psi_g"] \\
\mathcal F \arrow[r,"f"'] & \mathcal G
\end{tikzcd}
\]
\end{enumerate}
\end{definition}

Henceforth, we shall refer to the objects of ${\rm QCoh}_G(X)$ as $G$-sheaves and to the arrows as $G$-morphisms or $G$-maps. The collection $\{\phi_g\}_{g\in G}$ defining the $G$-structure on $\mathcal{F}$ will be referred to as its $G$-linearization and shall be suppressed from the notation when clear from the context.

\begin{lemma}\label{lem:abcat}
${\rm QCoh}_G(X)$ is an abelian category.
\end{lemma}
\begin{proof}
This is straightforward and well known, see for example \cite[1.4]{MR921490}; we indicatively verify that $\ker f$ and ${\rm coker }f$ of a $G$-morphism $f:(\mathcal{F},\{\phi_g\}_{g\in G})\rightarrow(\mathcal{G},\{\psi_g\}_{g\in G})$ live in ${\rm QCoh}_G(X)$: {since $\rho_g$ is an automorphism, $\rho_g^*$ is exact, so it commutes with kernels and images; as $\rho_g^*(f)=\psi_g^{-1}\circ f\circ\phi_g$, one has for all $g\in G$
\begin{eqnarray*}
&&\rho_g^*(\ker f)=\ker\rho_g^*(f)=\ker(\psi_g^{-1}\circ f\circ\phi_g)=\phi_g^{-1}(\ker f),\\
&&\rho_g^*({\rm Im }f)={\rm Im}\rho_g^*(f)={\rm Im}(\psi_g^{-1}\circ f\circ\phi_g)=\psi_g^{-1}({\rm Im} f).
\end{eqnarray*}
Hence $\phi_g$ restricts to an isomorphism $\rho_g^*\ker(f)\rightarrow\ker f$ and $\psi_g$ restricts to an isomorphism $\rho_g^*{\rm Im}(f)\rightarrow{\rm Im} f$; the cocycle identities descend by restriction. In particular $\phi_g$ restricted to $\rho_g^*\ker(f)$ is a $G$-linearization of $\ker f$, and $\psi_g$ induces a well-defined linearization
$
\rho^*_g{\rm coker}f=\rho^*_g\left(\mathcal{G}\right)/\rho_g^*\left({\rm Im}f\right)\rightarrow \mathcal{G}/{\rm Im}f={\rm coker}f.
$}
\end{proof}

\subsection{Equivariant functors and lifting criteria}
Let $\mathsf{For}:{\rm QCoh}_G(X)\rightarrow {\rm QCoh}(X)$ denote the forgetful functor. A functor $\mathsf{A}:{\rm QCoh}(X)\rightarrow {\rm QCoh}(X)$ will be called $G$-equivariant, or simply a $G$-functor, if there exists a functor $\mathsf{A}_G:{\rm QCoh}_G(X)\rightarrow {\rm QCoh}_G(X)$ such that the following diagram commutes
\[
\begin{tikzcd}[column sep=huge]
{\rm QCoh}_G(X) \arrow[r, "\mathsf{A}_G"] \arrow[d, "\mathsf{For}"] &{\rm QCoh}_G(X)\arrow[d, "\mathsf{For}"] \\
{\rm QCoh}(X) \arrow[r, "\mathsf{A}"] &{\rm QCoh}(X).
\end{tikzcd}
\]
{We do not require $\mathsf{A}$ to be additive: the exterior powers $\bigwedge^p(-)$, $p\geq 2$, which are our main example, are not additive.} In that case, one says that $\mathsf{A}$ lifts to $\mathsf{A}_G$ and the same terminology will be used if either one of the domain or target categories is replaced by ${\rm Vect}_k$ and the lifting is from or to $kG$-mod.

\begin{proposition}\label{prop:G-functors}
Let
$
\mathsf{A}:{\rm QCoh}(X)\rightarrow {\rm QCoh}(X), \;\mathsf{B}:{\rm QCoh}(X)\rightarrow {\rm Vect}_k,\;\mathsf{C}:{\rm Vect}_k\rightarrow {\rm QCoh}(X)
$
be {(not necessarily additive)} functors.
\begin{enumerate}
\item
$\mathsf{A}$ lifts to a functor $\mathsf{A}_G:{\rm QCoh}_G(X)\rightarrow {\rm QCoh}_G(X)$ if for each $g\in G$ there exists a natural isomorphism of functors $\sigma_g:\rho_g^*\circ \mathsf{A}\rightarrow \mathsf{A}\circ \rho_g^* $ satisfying $\sigma_{gh}=\sigma_h(\rho^*_g)\circ \rho_h^*(\sigma_g)$.\\

\item $\mathsf{B}$ lifts to a functor $\mathsf{B}_G:{\rm QCoh}_G(X)\rightarrow kG\text{-}{\rm mod}$ if for each $g\in G$ there exists a natural isomorphism of functors $\tau_g:\mathsf{B}\rightarrow \mathsf{B}\circ\rho^*_g$ satisfying $\tau_{gh}=\tau_h(\rho_g^*)\circ\tau_g$.\\

\item $\mathsf{C}$ lifts to a functor $\mathsf{C}_G:kG\text{-}{\rm mod}\rightarrow {\rm QCoh}_G(X)$ if for each $g\in G$ there exists a natural isomorphism of functors $\upsilon_g:\rho_g^*\mathsf{C}\rightarrow \mathsf{C}$ satisfying {$\upsilon_{gh}=\upsilon_h\circ\rho_h^*(\upsilon_g)$; the lift sends $(V,\{\theta_g\}_{g\in G})$ to $\mathsf{C}(V)$ with the $G$-linearization $\mathsf{C}(\theta_g)\circ\upsilon_g(V)$}.
\end{enumerate}
\end{proposition}
\begin{proof}
Let $(\mathcal{F},\{\phi_g\}_{g\in G}),(\mathcal{F'},\{\phi'_g\}_{g\in G})\in{\rm QCoh}_G(X)$ and let $f:\mathcal{F}\rightarrow\mathcal{F'}$ be a $G$-morphism.\\

(1): The isomorphism $\psi_g=\mathsf{A}(\phi_g)\circ\sigma_g(\mathcal{F}):\rho_g^* \mathsf{A}(\mathcal{F})\rightarrow \mathsf{A}(\mathcal{F})$ is a $G$-linearization on $\mathsf{A}(\mathcal{F})$, since
\[
\psi_{gh}
=\mathsf{A}(\phi_h)\circ \mathsf{A}\rho_h^*(\phi_g)\circ\sigma_h(\rho_g^*\mathcal{F})\circ\rho_h^*\sigma_g(\mathcal{F})
=\mathsf{A}(\phi_h)\circ \sigma_h(\mathcal{F})\circ \rho_h^*\mathsf{A}(\phi_g)\circ\rho_h^*\sigma_g(\mathcal{F})
=\psi_h\circ\rho^*_h(\psi_g),
\]
where the second equality follows from naturality of $\sigma_h$
\[
\begin{tikzcd}
\rho_h^*\mathsf{A}(\rho_g^*\mathcal{F}) \arrow[r,"\sigma_h(\rho_g^*\mathcal{F})"] \arrow[d,"\rho_h^*\mathsf{A}(\phi_g)"'] & \mathsf{A}\rho_h^*(\rho_g^*\mathcal F) \arrow[d,"\mathsf{A}\rho^*_h(\phi_g)"] \\
\rho_h^*\mathsf{A}(\mathcal F) \arrow[r,"\sigma_h(\mathcal{F})"] & \mathsf{A}\rho_h^*(\mathcal F).
\end{tikzcd}
\]

If $\psi'_g=\mathsf{A}(\phi'_g)\circ\sigma_g(\mathcal{F'})$ is the $G$-linearization of $\mathsf{A}(\mathcal{F'})$, then $\mathsf{A}(f)$ is a $G$-morphism, since
\[
\psi'_g\circ\rho^*_g(\mathsf{A}(f))=\mathsf{A}(\phi'_g)\circ\sigma_g(\mathcal{F'})\circ\rho^*_g(\mathsf{A}(f))=\mathsf{A}(\phi'_g)\circ \mathsf{A}\rho_g^*(f)\circ\sigma_g(\mathcal{F})=\mathsf{A}(f)\circ \mathsf{A}(\phi_g)\circ\sigma_g(\mathcal{F})=\mathsf{A}(f)\circ\psi_g,
\]
by naturality of $\sigma_g$.\\

(2): The $k$-isomorphism $\theta_g=\mathsf{B}(\phi_g)\circ\tau_g(\mathcal{F}):\mathsf{B}(\mathcal{F})\rightarrow \mathsf{B}(\mathcal{F})$ makes $\mathsf{B}(\mathcal{F})$ a $kG$-module $\mathsf{B}(\mathcal{F})$, since
\[
\theta_{gh}
=\mathsf{B}(\phi_{gh})\circ\tau_{gh}(\mathcal{F})
=\mathsf{B}(\phi_h)\circ \mathsf{B}\rho_h^*(\phi_g)\circ\tau_h(\rho_g^*\mathcal{F})\circ \tau_g(\mathcal{F})
=\mathsf{B}(\phi_h)\circ \tau_h(\mathcal{F})\circ \mathsf{B}(\phi_g)\circ\tau_g(\mathcal{F})
=\theta_h\theta_g,
\]
where naturality of $\tau_g$ has been used as follows
\[
\begin{tikzcd}
\mathsf{B}\rho_g^*(\mathcal{F}) \arrow[r,"\tau_h(\rho_g^*\mathcal{F})"] \arrow[d,"\mathsf{B}(\phi_g)"'] & \mathsf{B}\rho_h^*\rho_g^*(\mathcal F) \arrow[d,"\mathsf{B}\rho_h^*(\phi_g)"] \\
\mathsf{B}(\mathcal F) \arrow[r,"\tau_h(\mathcal{F})"] & \mathsf{B}\rho_h^*(\mathcal F).
\end{tikzcd}
\]

If $\theta_g'$ is the $G$-action on $\mathsf{B}(\mathcal{F}')$, then $\mathsf{B}(f)$ is a morphism of $kG$-modules, since
\[
\theta_g'\circ \mathsf{B}(f)
=\mathsf{B}(\phi_g')\circ\tau_g(\mathcal{F'})\circ \mathsf{B}(f)
=\mathsf{B}(\phi_g')\circ \mathsf{B}(\rho_g^*(f))\circ\tau_g(\mathcal{F})
=\mathsf{B}(f)\circ \mathsf{B}(\phi_g)\circ\tau_g(\mathcal{F})
=\mathsf{B}(f)\circ \theta_g.
\]
Note that the second equality above is due to naturality of $\tau_g$ and the third due to $G$-equivariance of $f$.\\

{(3): Let $(V,\{\theta_g\}_{g\in G})\in kG\text{-}{\rm mod}$ and set
\[
\psi_g=\mathsf{C}(\theta_g)\circ\upsilon_g(V):\rho_g^*\mathsf{C}(V)\rightarrow\mathsf{C}(V).
\]
Naturality of $\upsilon_h$ gives $\mathsf{C}(\theta_g)\circ\upsilon_h(V)=\upsilon_h(V)\circ\rho_h^*\mathsf{C}(\theta_g)$, and so, using $\theta_{gh}=\theta_h\theta_g$ and the cocycle condition for $\upsilon$,
\[
\psi_{gh}=\mathsf{C}(\theta_h)\mathsf{C}(\theta_g)\upsilon_h(V)\rho_h^*(\upsilon_g(V))
=\mathsf{C}(\theta_h)\upsilon_h(V)\rho_h^*\mathsf{C}(\theta_g)\rho_h^*(\upsilon_g(V))
=\psi_h\circ\rho_h^*(\psi_g).
\]
Thus $\{\psi_g\}_{g\in G}$ is a $G$-linearization of $\mathsf{C}(V)$. If $f:V\rightarrow W$ is a $kG$-module homomorphism, then naturality of $\upsilon_g$ together with $f\theta^V_g=\theta^W_gf$ shows that $\mathsf{C}(f)$ is $G$-equivariant.}
\end{proof}

\noindent {\em Remark.}
 It is often the case, {\cite[2.8]{MR3979084}}, that the lifting criteria above are {taken as definitions of $G$-equivariant functors}.

\subsection{Fundamental equivariant operations}
Tensor products, exterior powers and global sections are known to be $G$-functors; see for example {\cite[p.32]{MR1304906}} or \cite[p.94]{MR2665168}. This can be verified by directly showing they satisfy the criteria Proposition \ref{prop:G-functors}.
\begin{proposition}\label{prop:exGfunctors}
Let $(\mathcal{F},\{\phi_g\}_{g\in G})\in {\rm QCoh}_G(X)$. The following functors are $G$-functors.
\begin{enumerate}
\item
$ -\otimes \mathcal{F}:{\rm QCoh}(X)\rightarrow{\rm QCoh}(X)$, for any $ (\mathcal{F},\{\phi_g\}_{g\in G})\in {\rm QCoh}_G(X)$.\\
\item
$\bigwedge^p(-):{\rm QCoh}(X)\rightarrow {\rm QCoh}(X)$, for any $p\in\mathbb{Z}_{\geq 0}$.\\
\item
$H^0(-):{\rm QCoh}(X)\rightarrow{\rm Vect}_k$, where $H^0(-)=H^0(X,-)=\Gamma(X,-)$.\\
\item
$-\otimes\mathcal{F}:{\rm Vect}_k\rightarrow{\rm QCoh}(X)$, for any $ (\mathcal{F},\{\phi_g\}_{g\in G})\in {\rm QCoh}_G(X)$.
\end{enumerate}
\end{proposition}

\begin{proof}
(1): If $(\mathcal{F},\{\phi_g\}_{g\in G})\in {\rm QCoh}_G(X)$ and $\kappa_g:\rho_g^*(-\otimes \mathcal{F})\rightarrow\rho_g^*(-)\otimes \rho_g^*\mathcal{F}$ is the natural isomorphism of the compatibility of tensor products with pullbacks \cite[II.5.16(e)]{Hartshorne:77}, then one has a commutative triangle
\[
\begin{tikzcd}
\rho_{gh}^*(-\otimes \mathcal{F}) \arrow[rr, "\kappa_{gh}"] \arrow[dr,"\rho_h^*(\kappa_g)"'] & &
\rho^*_{gh}(-)\otimes\rho^*_{gh}(\mathcal{F}) \\
& \rho_h^*(\rho_g^*(-)\otimes\rho_g^*(\mathcal{F}) )  \arrow[ur,"\kappa_h"'] &
\end{tikzcd}
\]
and thus the natural isomorphism
\[
\sigma_g:\rho_g^*(-\otimes \mathcal{F})\overset{\kappa_g}{\longrightarrow}\rho_g^*(-)\otimes \rho_g^*\mathcal{F}\overset{{\rm id}\otimes\phi_g}{\longrightarrow}\rho_g^*(-)\otimes \mathcal{F}
\]
satisfies the required condition
\begin{eqnarray*}
\sigma_{gh}
&=&{\rm id}\otimes \phi_{gh}\circ\kappa_{gh}
=
{\rm id} \otimes (\phi_h\circ \rho_h^*(\phi_g))\circ (\kappa_h\circ \rho_h^*(\kappa_g))
=
{\rm id}\otimes\phi_h \circ {\rm id}\otimes\rho_h^*(\phi_g)\circ \kappa_h\circ\rho_h^*(\kappa_g)
\\
&=&{\rm id}\otimes\phi_h \circ \kappa_h(\rho_g^*)\circ\rho_h^*({\rm id}\otimes\phi_g)\circ\rho_h^*(\kappa_g)
=\sigma_h(\rho_g^*)\circ\rho_h^*(\sigma_g).
\end{eqnarray*}
Hence $-\otimes \mathcal{F}:{\rm QCoh}(X)\rightarrow{\rm QCoh}(X)$ lifts to a functor ${\rm QCoh}_G(X)\rightarrow{\rm QCoh}_G(X)$.\\

(2):  Similarly, for any $p\in\mathbb{Z}_{p\geq 0}$, one uses the natural isomorphism of the compatibility of exterior powers with pullbacks \cite[7.8.(3)]{MR4225278},
\[
\sigma_g:\rho_g^*\big(\bigwedge^p(-)\big)
\rightarrow
\bigwedge^p\rho_g^*(-),
\]
to verify that $\bigwedge^p(-):{\rm QCoh}(X)\rightarrow {\rm QCoh}(X)$ lifts to a functor ${\rm QCoh}_G(X)\rightarrow{\rm QCoh}_G(X)$.\\

(3): Let $\{c_g\}_{g\in G}$ be the $G$-linearization on the structure sheaf $\mathcal{O}_X$. Suppressing the curve $X$ from the notation, consider the global sections functor $H^0(-):{\rm QCoh}(X)\rightarrow {\rm Vect}_k$. If $(\mathcal{F},\{\phi_g\})_{g\in G}\in {\rm QCoh}_G(X)$, then a global section $s\in H^0(\mathcal{F})$ is a morphism $s:\mathcal{O}_X\rightarrow \mathcal{F}$, and the composition
\[
\mathcal{O}_X\overset{c_g^{-1}}{\rightarrow}\rho_g^*\mathcal{O}_X\overset{\rho_g^*(s)}{\rightarrow}\rho_g^*\mathcal{F}
\]
gives an element of $H^0(\mathcal{\rho_g^*F})$. For each $g\in G$, the $k$-linear map
\[
\tau_g(\mathcal{F}):H^0(\mathcal{F})\rightarrow H^0(\rho_g^*\mathcal{F}),\;s\mapsto\rho_g^*(s)\circ c_g^{-1}
\]
defines a natural transformation $\tau_g:H^0(-)\rightarrow H^0(\rho_g^*-)$: if $f:\mathcal{F}\rightarrow \mathcal{G}$ is a morphism in ${\rm QCoh}_G(X)$, then for all $s\in H^0(\mathcal{F})$ one obtains
\[
(\tau_g(\mathcal{G})\circ H^0(f))(s)
=\tau_g(\mathcal{G})(f\circ s)
=\rho_g^*(f\circ s)\circ c_g^{-1}
=\rho_g^*(f)\circ \rho_g^*(s)\circ c_g^{-1}
=(H^0(\rho_g^*(f))\circ\tau_g(\mathcal{F}))(s).
\]
To verify the condition $\tau_{gh}=\tau_h(\rho_g^*)\circ\tau_g$ one evaluates at $\mathcal{F}\in {\rm QCoh}_G(X)$, then at $s\in H^0(\mathcal{F})$
\[
(\tau_h(\rho^*_g\mathcal{F})\circ \tau_g(\mathcal{F}))(s)
=\tau_h(\rho^*_g\mathcal{F})(\rho_g^*(s)\circ c_g^{-1})
=\rho_h^*\rho_g^*(s)\circ \rho_h^*(c_g^{-1})\circ c_h^{-1}
=\rho_{gh}^*(s)\circ c_{gh}^{-1}
=\tau_{gh}(\mathcal{F})(s).
\]
So $H^0(-)$ is a $G$-functor.\\

(4): If $(\mathcal{F},\{\phi_g\}_{g\in G})\in {\rm QCoh}_G(X)$ and $V \in {\rm Vect}_k$, then the isomorphisms
\[
\upsilon_{g}(\mathcal{F})
={\rm id}_V\otimes \phi_g:V\otimes\rho_g^*(\mathcal{F})\rightarrow V\otimes \mathcal{F}
\]
are clearly natural in $V$ and thus, $\upsilon_g:\rho_g^*(-\otimes\mathcal{F})=-\otimes\rho_g^*(\mathcal{F})\rightarrow -\otimes\mathcal{F}$ is a natural isomorphism, satisfying
\[
\upsilon_{gh}=
{\rm id}_V\otimes\phi_{gh}={\rm id}_V\otimes (\phi_h\circ \rho_h^*(\phi_g))=({\rm id}_V\otimes\phi_h)\circ ({\rm id}_V\otimes\rho_h^*(\phi_g))=\upsilon_h\circ\rho_h^*(\upsilon_g),
\]
meaning that $-\otimes\mathcal{F}:{\rm Vect}_k\rightarrow{\rm QCoh}(X)$ is a $G$-functor. {By part (3) of Proposition~\ref{prop:G-functors}, its lift sends $(V,\{\theta_g\}_{g\in G})$ to $V\otimes\mathcal{F}$ with the diagonal $G$-linearization $\theta_g\otimes\phi_g$; for $\mathcal{F}=\mathcal{O}_X$ this is $\theta_g\otimes c_g$.}
\end{proof}

It is worth noting that by \cite[2.19]{MR3979084}, an adjunction between functors that satisfy the $G$-compatibility conditions of Proposition \ref{prop:G-functors} lifts canonically to an adjunction between their $G$-equivariant counterparts. We give a self-contained proof in the special case of the adjunction $-\otimes\mathcal{O}_X\dashv H^0(-)$ as functors between the categories ${\rm Vect_k}$ and ${\rm QCoh}(X)$,  \cite[7.11]{MR4225278}.
\begin{proposition}\label{prop:G-adj}
The functor $-\otimes\mathcal{O}_X:kG\text{-mod}\rightarrow {\rm QCoh}_G(X)$ is left adjoint to the functor $H^0(-):{\rm QCoh}_G(X)\rightarrow kG\text{-mod}$.
\end{proposition}
\begin{proof}
Let $(\mathcal{F},\{\phi_g\}_{g\in G})\in {\rm QCoh}_G(X)$, let $(V,\{\theta_g\}_{g\in G})\in kG\text{-mod}$, and let
$\{c_g\}_{g\in G}$ be the $G$-linearization of the structure sheaf $\mathcal{O}_X$. Recall by the examples above that the $G$-action on $H^0(\mathcal{F})$ is given by $\eta_g(s)=\phi_g\circ \rho_g^*(s)\circ c_g^{-1}$ and that the $G$-linearization on $V\otimes\mathcal{O}_X$ is given by $\kappa_g=\theta_g\otimes c_g$. Accordingly, the $G$-action on $H^0(V\otimes \mathcal{O}_X)$ induced by that $G$-linearization is given by $\zeta_g(s)=({\theta_g}\otimes c_g)\circ \rho_g^*(s)\circ c_g^{-1}$.\\

If $f:V\otimes\mathcal{O}_X\rightarrow \mathcal{F}$ is a $G$-morphism, then $\phi_g\circ\rho_g^*(f)=f\circ\kappa_g=f\circ(\theta_g\otimes c_g)$. For $s\in H^0(V\otimes \mathcal{O}_X)$,
\[
(\eta_g\circ H^0(f))(s)
=\eta_g(f\circ s)=\phi_g\circ \rho_g^*(f\circ s)\circ c_g^{-1}
=f\circ(\theta_g\otimes c_g)\circ\rho_g^*(s)\circ c_g^{-1}=f\circ\zeta_g(s)=(H^0(f)\circ\zeta_g)(s)
\]
and thus $H^0(f)$ is a $kG$-module homomorphism.\\

Conversely, the non-equivariant adjunction bijection ${\rm Hom}_{{\rm QCoh}(X)}(V\otimes \mathcal{O}_X,\mathcal{F})\cong{\rm Hom}_k(V,H^0(\mathcal{F}))$ implies that any $k$-linear map $V\cong H^0(V\otimes\mathcal{O}_X)\rightarrow H^0(\mathcal{F})$ must be of the form $H^0(f)$ for some $f:V\otimes\mathcal{O}_X\rightarrow \mathcal{F}$. {Suppose that $H^0(f)$ is a $kG$-module homomorphism. For $v\in V$ let $s_v=v\otimes 1\in H^0(V\otimes\mathcal{O}_X)$, so that $\zeta_g(s_v)=s_{\theta_g v}$, and let $t_v=\rho_g^*(s_v)\circ c_g^{-1}$, a global section of $\rho_g^*(V\otimes\mathcal{O}_X)=V\otimes\rho_g^*\mathcal{O}_X$. The computations above show that $\eta_g(f\circ s_v)=f\circ\zeta_g(s_v)$ is equivalent to
\[
\phi_g\circ \rho_g^*(f)\circ t_v=f\circ(\theta_g\otimes c_g)\circ t_v .
\]
Thus the two $\mathcal{O}_X$-linear maps $\phi_g\circ \rho_g^*(f)$ and $f\circ(\theta_g\otimes c_g)$ from $V\otimes\rho_g^*\mathcal{O}_X$ to $\mathcal{F}$ agree on the global sections $t_v=v\otimes c_g^{-1}(1)$, $v\in V$, which generate $V\otimes\rho_g^*\mathcal{O}_X$ as an $\mathcal{O}_X$-module, since $c_g$ is an isomorphism. Hence} $\phi_g\circ\rho_g^*(f)=f\circ\kappa_g=f\circ(\theta_g\otimes c_g)$, and so $f$ is a $G$-morphism. {No inverse of a section is needed; the non-equivariant adjunction bijection therefore restricts to ${\rm Hom}_{{\rm QCoh}_G(X)}(V\otimes\mathcal{O}_X,\mathcal{F})\cong{\rm Hom}_{kG}(V,H^0(\mathcal{F}))$.}
\end{proof}

\subsection{Contraction maps}
Let $\mathcal{F}\in{\rm QCoh}(X)$ and let $p\in\mathbb{Z}_{{\geq 1}}$ {(for $p=0$ we set $\bigwedge^{-1}\mathcal{F}=0$, so that the maps below are zero)}. Any permutation $\sigma\in\mathfrak{S}_p$ in the symmetric group defines a unique natural isomorphism $\tau_\sigma: \mathcal{F}^{\otimes p} \rightarrow \mathcal{F}^{\otimes p}$, often referred to as the {\em braiding} morphism, which acts by permuting the tensor factors. For $1\leq i\leq p$ let $\sigma_i\in\mathfrak{S}_p$ be the cycle $(i,i+1,\ldots,p)$ and define
\[
{\rm Alt}_p=\sum_{i=1}^p(-1)^{p-i} \tau_{\sigma_i} :\mathcal{F}^{\otimes p} \longrightarrow \mathcal{F}^{\otimes p}.
\]
{If $\mathcal{F}\in{\rm QCoh}_G(X)$, then each $\tau_\sigma$, and hence ${\rm Alt}_p$, is $G$-equivariant. Let $\pi_p:\mathcal{F}^{\otimes p}\twoheadrightarrow \bigwedge^p\mathcal{F}$ denote the natural quotient map. The map ${\rm Alt}_p$ itself is not alternating (e.g. ${\rm Alt}_3(v\otimes v\otimes v)=v\otimes v\otimes v$), but the composite $(\pi_{p-1}\otimes{\rm id})\circ{\rm Alt}_p:\mathcal{F}^{\otimes p}\rightarrow\bigwedge^{p-1}\mathcal{F}\otimes\mathcal{F}$ is: if two local sections $s_j=s_l$, $j<l$, coincide, every summand that omits neither of them vanishes after applying $\pi_{p-1}$, and the two remaining summands, for $i=j$ and $i=l$, cancel. (This proves alternation, and not merely skew-symmetry, so it is valid also in characteristic two.) Thus} the universal property of exterior powers implies that there exists a unique $G$-morphism $\Delta_p$ that makes the following diagram commute.
\[
\begin{tikzcd}[column sep=huge]
\mathcal{F}^{\otimes p} \arrow[r, "{\rm Alt}_p"] \arrow[d, "\pi_p"] & \mathcal{F}^{\otimes p-1} \otimes \mathcal{F} \arrow[d, "\pi_{p-1} \otimes \text{id}"] \\
\bigwedge^p \mathcal{F} \arrow[r, "\Delta_p"] & \bigwedge^{p-1} \mathcal{F} \otimes \mathcal{F}.
\end{tikzcd}
\]

It should be noted that locally on the level of sections $\Delta_p$ is given by
\[
s_1\wedge\cdots\wedge s_p\mapsto \sum_{i=1}^{p}(-1)^{p-i}s_1\wedge\cdots\wedge s_{i-1}\wedge s_{i+1}\wedge\cdots \wedge s_p \otimes s_i.
\]
\begin{definition}\label{def:contr}
The {\em contraction map} associated to a morphism $f:\mathcal{F}\rightarrow\mathcal{G}$ in ${\rm QCoh}(X)$ is
\[
\delta_{f,p}:\bigwedge^p\mathcal{F}\rightarrow\bigwedge^{p-1}\mathcal{F}\otimes \mathcal{G},\;\delta_{f,p}=({\rm id}_{\bigwedge^{p-1}\mathcal{F}}\otimes f)\circ \Delta_p.
\]
\end{definition}

By construction, contraction maps associated to $G$-morphisms are $G$-morphisms themselves. This allows us to extend a fundamental ingredient of the construction of Koszul complexes to the $G$-equivariant setting.

\begin{proposition}\label{prop:contrprop}
A short exact sequence of vector bundles in ${\rm QCoh}_G(X)$
\[
0\rightarrow \mathcal{H}\overset{g}{\rightarrow}\mathcal{F}\overset{f}{\rightarrow}\mathcal{G}\rightarrow 0
\]
with ${\rm rank}(\mathcal{G})=1$, gives rise to a short exact sequence in ${\rm QCoh}_G(X)$
\[
0\rightarrow \bigwedge^{p}\mathcal{H}\overset{\wedge^p(g)}{\longrightarrow}\bigwedge^{p}\mathcal{F}\overset{\delta_{f,p}}{\longrightarrow}\bigwedge^{p-1}\mathcal{H}\otimes\mathcal{G}\rightarrow 0.
\]
\end{proposition}

\begin{proof}
The contraction map $\delta_{f,p}$ is a $G$-map, and by Lemma \ref{lem:abcat} its kernel, image and cokernel are in ${\rm QCoh}_G(X)$. It thus suffices to prove the statement in the non-equivariant case; this is well known and can be found in \cite[3.2]{MR1644323} or \cite[4.13]{MR1335917}. Indicatively, its proof goes as follows: if $\mathcal{U}=\{U_a\}_{a\in A}$ is an affine open cover of $X$, then one has, for all $a\in A$, split short exact sequences of $\mathcal{O}_X(U_a)$-modules
$0\rightarrow \mathcal{H}\mid_{U_a}\overset{g\mid_{U_a}}{\rightarrow}\mathcal{F}\mid_{U_a}\overset{f\mid_{U_a}}{\rightarrow}\mathcal{G}\mid_{U_a}\rightarrow 0$. Then
\[
\bigwedge^p\mathcal{F}\mid_{U_a}\cong \bigoplus_{m+n=p}\bigwedge^{m}\mathcal{H}\mid_{U_a}\otimes_{\mathcal{O}_X(U_a)}\bigwedge^{n}\mathcal{G}\mid_{U_a}
\cong
\bigwedge^{p}\mathcal{H}\mid_{U_a}
\oplus
 \bigwedge^{p-1}\mathcal{H}\mid_{U_a}\otimes_{\mathcal{O}_X(U_a)}\mathcal{G}\mid_{U_a},
\]
since $\bigwedge^{n}\mathcal{G}\mid_{U_a}=0$ for all $n>{\rm rank}(\mathcal{G})=1$, with $\wedge^p(g)\mid_{U_a}$ the inclusion of the first factor and $\delta_{f,p}\mid_{U_a}$ the projection onto the second. Hence $\ker \delta_{f,p}\cong \bigwedge^{p}\mathcal{H}\text{ and }{\rm Im}{\delta_{f,p}}=\bigwedge^{p-1}\mathcal{H}\otimes\mathcal{G}.
$
\end{proof}

\subsection{The Koszul complex}
Let $L$ be a very ample line bundle on $X$, and let
\[
V=H^0(L)=H^0(X,L)\text{ and }S={\rm Sym}(V)=\bigoplus_{q\geq 0}{\rm Sym}^q(V).
\]
\begin{definition}\label{def:Koszul}
{Let $A(X,L)=\bigoplus_{q\geq 0}H^0(X,L^{\otimes q})$ be the section ring, a graded $S$-module. The Koszul complex associated to the pair $(X,L)$ is the complex $\bigwedge^\bullet V\otimes A(X,L)$ of graded $S$-modules; its graded strand of total degree $n=p+q$ is the following complex of $k$-vector spaces}
\[
\cdots
\overset{d_{p+2,q-2}}{\longrightarrow}
\bigwedge^{p+1}V\otimes H^0(L^{\otimes (q-1)})
\overset{d_{p+1,q-1}}{\longrightarrow}
\bigwedge^{p}V\otimes H^0(L^{\otimes q})\overset{d_{p,q}}{\longrightarrow}
\bigwedge^{p-1}V\otimes H^0(L^{\otimes (q+1)})\overset{d_{p-1,q+1}}{\longrightarrow}\cdots
\]
with differentials given by
\[
d_{p,q}(s_1\wedge\cdots \wedge s_p\otimes f)=\sum_{i=1}^p(-1)^{p-i}s_1\wedge\cdots\wedge s_{i-1}\wedge s_{i+1}\wedge \cdots \wedge s_p\otimes (fs_i).
\]
Its cohomology $ \ker d_{p,q}/{\rm Im}\;d_{p+1,q-1}$ will be denoted by $K_{p,q}(X,L)$ and referred to as the $(p,q)$-th Koszul cohomology of the pair $(X,L)$.
\end{definition}

The Koszul complex can be built functorially from the evaluation map ${\rm ev}_L:V\otimes\mathcal{O}_X\rightarrow L$ which on the level of sections is given by $s\otimes f\mapsto fs$, see \cite[\S2.1]{MR2573635}. One then obtains formulas for the dimensions of $K_{p,q}(X,L)$ over $k$ in terms of that of the Zariski cohomology of $L$, its tensor powers and the kernel of the evaluation map, see \cite[Proposition 2.5]{MR2573635} or \cite[Proposition 3.2]{MR3729076}. This section revisits the construction to verify that if $L\in{\rm QCoh}_G(X)$, then every step is $G$-equivariant; our aim is to generalize the stated dimension formulas into relationships between the induced $G$-representations.\\

\subsection{Equivariant Koszul cohomology}
Let $L\in{\rm QCoh}_G(X)$ be a very ample line bundle on $X$, so that by Proposition \ref{prop:exGfunctors}, $V=H^0(L)$ is a $kG$-module and $V\otimes \mathcal{O}_X\in{\rm QCoh}_G(X)$. Let
\[
\Phi:{\rm End}_{kG}(V)\rightarrow{\rm Hom}_{{\rm QCoh}_G(X)}(V\otimes \mathcal{O}_X,L)
\]
be the bijection associated to the adjunction $-\otimes \mathcal{O}_X\dashv H^0(-)$ of Proposition \ref{prop:G-adj}.

\begin{definition}
The map $\Phi({\rm id}_V):V\otimes\mathcal{O}_X\rightarrow L$ will be denoted by ${\rm ev}_L$ or simply ${\rm ev}$ and referred to as the evaluation map associated to $L$. Its kernel will be denoted by $\mathcal{M}_L$ and referred to as the kernel bundle or Lazarsfeld bundle.
\end{definition}
In what follows, we write $[W]$ for the virtual representation corresponding to a $kG$-module $W$.

\begin{theorem}\label{th:main}
The Koszul complex associated to a pair $(X,L)$ with $L$ very ample and $L\in{\rm QCoh}_G(X)$ is a complex of $kG$-modules; the virtual representation of its cohomology satisfies the identity
\[
[K_{p,q}(X,L)]
=[H^0\big(\bigwedge^p\mathcal{M}_L\otimes L^{\otimes q}\big)]
-[H^0\big(\bigwedge^{p+1}V\otimes  L^{\otimes (q-1)}\big)]
+[H^0\big( \bigwedge^{p+1}\mathcal{M}_L\otimes L^{\otimes (q-1)}\big)].
\]
\end{theorem}
\begin{proof}
By construction, ${\rm ev}_L$ is a $G$-morphism, and by Lemma \ref{lem:abcat}, $\mathcal{M}_L\in{\rm QCoh}_G(X)$. Thus
\[
0\rightarrow \mathcal{M}_L\overset{i}{\rightarrow} V\otimes \mathcal{O}_X\overset{{\rm ev}}{\rightarrow} L\rightarrow 0
\]
is a short exact sequence in ${\rm QCoh}_G(X)$ and by Proposition \ref{prop:contrprop}, the same is true for
\[
0\rightarrow
\bigwedge^p\mathcal{M}_L
\overset{\wedge^p(i)}{\rightarrow}
\bigwedge^p(V\otimes \mathcal{O}_X)
\overset{\delta_{{\rm ev},p}}{\rightarrow}\bigwedge^{p-1} \mathcal{M}_L \otimes L\rightarrow 0,
\]
where $\delta_{{\rm ev},p}$ is the contraction map of Definition \ref{def:contr} associated to the evaluation map. Applying the exact $G$-functor $-\otimes L^{\otimes q}$, see Proposition \ref{prop:exGfunctors}, and setting $f_{p,q}=\wedge^p(i)\otimes{\rm id}_{L^{\otimes q}},\;g_{p,q}=\delta_{{\rm ev},p}\otimes {\rm id}_{L^{\otimes q}}$ gives that
\[
0\rightarrow
\bigwedge^{p}\mathcal{M}_L\otimes L^{\otimes q}
\overset{f_{p,q}}{\longrightarrow}
\bigwedge^{p}V\otimes L^{\otimes q}
\overset{g_{p,q}}{\longrightarrow}
\bigwedge^{p-1}\mathcal{M}_L\otimes L^{\otimes (q+1)}
\rightarrow 0
 \tag{*}\label{eq:ses1}
\]
is also a short exact sequence in ${\rm QCoh}_G(X)$. The compositions
\[
\delta_{p,q}:
\bigwedge^{p}V\otimes L^{\otimes q}
\overset{ g_{p,q}}{\longrightarrow}
\bigwedge^{p-1}\mathcal{M}_L\otimes L^{\otimes (q+1)}
\overset{f_{p-1,q+1}}{\longrightarrow}
\bigwedge^{p-1}V\otimes L^{\otimes q+1};
\]
are $G$-morphisms, and thus Proposition \ref{prop:exGfunctors} implies that $d_{p,q}=H^0(\delta_{p,q})$ are maps of $kG$-modules. They agree with the Koszul differentials of Definition \ref{def:Koszul} and thus the complex of $kG$-modules
\[
\cdots
\overset{d_{p+2,q-2}}{\longrightarrow}
\bigwedge^{p+1}V\otimes H^0(L^{\otimes (q-1)})
\overset{d_{p+1,q-1}}{\longrightarrow}
\bigwedge^{p}V\otimes H^0(L^{\otimes q})\overset{d_{p,q}}{\longrightarrow}
\bigwedge^{p-1}V\otimes H^0(L^{\otimes (q+1)})\overset{d_{p-1,q+1}}{\longrightarrow}\cdots
\]
is exactly the Koszul complex as defined above.\\

One then applies the left exact $G$-functor $H^0(-)$ to (\ref{eq:ses1}) to obtain the exact sequence
\[
0\rightarrow
H^0\big(\bigwedge^{p}\mathcal{M}_L\otimes L^{\otimes q}\big)
\overset{H^0(f_{p,q})}{\longrightarrow}
H^0\big(\bigwedge^{p}V\otimes L^{\otimes q}\big)
\overset{H^0(g_{p,q})}{\longrightarrow}
 H^0\big(\bigwedge^{p-1}\mathcal{M}_L\otimes L^{\otimes (q+1)}\big)
\]
and computes
\begin{align*}
\ker d_{p,q}&=\ker H^0(f_{p-1,q+1}\circ g_{p,q})=\ker H^0(g_{p,q})={\rm Im}H^0( f_{p,q})\cong H^0\big(\bigwedge^{p}\mathcal{M}_L\otimes L^{\otimes q}\big)\\
{\rm Im}d_{p+1,q-1}&={\rm Im}H^0(f_{p,q}\circ g_{p+1,q-1})\cong{\rm Im} H^0(g_{p+1,q-1})\cong \bigwedge^{p+1}V\otimes H^0(L^{\otimes (q-1)})/H^0\big( \bigwedge^{p+1}\mathcal{M}_L\otimes L^{\otimes (q-1)}\big),\\
\end{align*}
to conclude that
\[
[K_{p,q}(X,L)]
=[H^0\big(\bigwedge^p\mathcal{M}_L\otimes L^{\otimes q}\big)]
-[H^0\big(\bigwedge^{p+1}V\otimes  L^{\otimes (q-1)}\big)]
+[H^0\big( \bigwedge^{p+1}\mathcal{M}_L\otimes L^{\otimes (q-1)}\big)].
\]
\end{proof}

\section{Explicit formulas for canonical curves}\label{sec:canonical}
Let $(X,\mathcal{O}_X)$ be a smooth, projective, irreducible{, non-hyperelliptic} curve of genus $g\geq 4$ over an algebraically closed field $k$, and let $G$ be a finite group of order coprime to the characteristic of $k$ acting faithfully on $X${, with the convention $\rho_{gh}=\rho_g\circ\rho_h$ of Section~\ref{sec:equiv}. Non-hyperellipticity guarantees that $\Omega_X$ is very ample, so that the results of Section~\ref{sec:equiv} apply to $L=\Omega_X$}.

\subsection{Equivariant structure of the canonical bundle} Let $L=\Omega_X$ be the canonical bundle on $X$ {and let $\{c_g\}_{g\in G}$ be the $G$-linearization of $\mathcal{O}_X$. By functoriality of K\"ahler differentials, each automorphism $\rho_g$ induces a canonical isomorphism
\[
a_g=d\rho_g:\rho_g^*\Omega_{X/k}\longrightarrow\Omega_{X/k},\qquad 1\otimes df\longmapsto d(\rho_g^\# f),
\]
defined on suitable open sets and extended $\mathcal{O}_X$-linearly; it is an isomorphism because $\rho_g$ is. The chain rule gives the cocycle condition $a_{gh}=a_h\circ\rho_h^*(a_g)$ of Definition~\ref{def:G-sheaf}.} Thus $\Omega_X\in{\rm QCoh}_G(X)$.

\subsection{Virtual Differences of Koszul Representations} For $p,q\in\mathbb{Z}_{\geq 0}$, let $K_{p,q}(X,\Omega_X)$ be the Koszul cohomology groups of Definition \ref{def:Koszul}. Then by \cite[3.6]{MR3729076} {$\dim_k K_{0,0}(X,\Omega_X)=1=\dim_k K_{g-2,3}(X,\Omega_X)$; as $kG$-modules, $K_{0,0}(X,\Omega_X)\cong k$ is trivial, while $K_{g-2,3}(X,\Omega_X)\cong D:=\det V=\bigwedge^gH^0(\Omega_X)$, see Proposition~\ref{prop:vanish}. Moreover,} $K_{p,q}(X,\Omega_X)=0$ for all other values of $(p,q)$ except when $q=1$ and $1\leq p\leq g-3$, or $q=2$ and $1\leq p\leq g-3$. This leads us to consider the following differences of virtual representations, mimicking the approach of calculating differences of $k$-vector space dimensions.

\begin{corollary}\label{cor:bettivanish}
Let $V=H^0(\Omega_X)$ and let $\mathcal{M}_{\Omega_X}=\ker{\rm ev}_{\Omega_X}$. For $1\leq p\leq {g-2}$,
\[
[K_{p,1}(X,\Omega_X)]-[K_{p-1,2}(X,\Omega_X)]=
[\bigwedge^{p}V\otimes V]-[\bigwedge^{p+1}V]
-[H^0\big(\bigwedge^{p-1}\mathcal{M}_{\Omega_X}\otimes \Omega_X^{\otimes 2}\big)].
\]
\end{corollary}
\begin{proof}
{Applying Theorem~\ref{th:main} with $(p,q)=(p,1)$ and noting that $H^0\big( \bigwedge^{p+1}\mathcal{M}_{\Omega_X}\big)$ vanishes for $p+1\leq g-1$, see \cite[Remark 2.6]{MR2573635}, gives
\[
[K_{p,1}]=[H^0(\bigwedge^{p}\mathcal{M}_{\Omega_X}\otimes\Omega_X)]-[\bigwedge^{p+1}V].
\]
Applying it with $(p-1,2)$ gives
\[
[K_{p-1,2}]=[H^0(\bigwedge^{p-1}\mathcal{M}_{\Omega_X}\otimes\Omega_X^{\otimes 2})]-[\bigwedge^{p}V\otimes V]+[H^0(\bigwedge^{p}\mathcal{M}_{\Omega_X}\otimes\Omega_X)],
\]
and subtracting cancels the common term. The argument is valid for every $1\leq p\leq g-2$.}
\end{proof}

{
\begin{remark}\label{rem:individual}
Corollary~\ref{cor:bettivanish} computes virtual differences only, but whenever one of the two terms on the left-hand side vanishes it computes the class of the other. Vanishing of a Koszul cohomology group is a property of the underlying vector space, so the classical (non-equivariant) vanishing theorems apply verbatim. Let $c$ be the Clifford index of $X$ and suppose that $K_{p,1}(X,\Omega_X)=0$ for $p\geq g-1-c$, as predicted by Green's conjecture~\cite{MR739785}; this holds, for instance, for general curves of each genus {in characteristic zero}~\cite{MR1941089,MR2157134}. By duality~\cite{MR2573635}, $\dim_k K_{p,2}(X,\Omega_X)=\dim_k K_{g-2-p,1}(X,\Omega_X)$, hence $K_{p,2}(X,\Omega_X)=0$ for $p<c$. Consequently the right-hand side of Corollary~\ref{cor:bettivanish} equals $[K_{p,1}(X,\Omega_X)]$ for $1\leq p\leq c$ and $-[K_{p-1,2}(X,\Omega_X)]$ for $g-1-c\leq p\leq {g-2}$; only the middle range $c<p<g-1-c$ is left undetermined. For the trigonal curves of Section~\ref{sec:Kummer} one has $c=1$, so the middle range is as large as possible; there the individual characters are obtained from Schreyer's description of the resolution (Theorem~\ref{Schreyer}).
\end{remark}
}

To calculate the right hand side of the above equation one employs the theory of $G$-equivariant Euler characteristics. To this end, we first review the results of Ellingsrud and L\o nsted from \cite{MR589254} which provide explicit decompositions of the virtual representations $[H^0(\mathcal{E})]-[H^1(\mathcal{E})]$ for $G$-equivariant vector bundles $\mathcal{E}$ on $X$. For a more recent treatment of the subject and generalizations, see \cite{MR2205171}.

\subsection{ Equivariant Euler characteristics.}

Let $\pi:X\rightarrow Y=X/G$ be the quotient map of the action of $G$ on $X$ and let $G_P={\rm Stab}_G(P)$ be the decomposition group of $P\in X$. It is a cyclic group of order equal to the ramification index of $P$; the latter is an invariant of the orbit of $P$ and thus may be denoted by $e_Q$, with $Q=\pi(P)\in Y$. Write $\{\chi_P^d:\;0\leq d\leq e_Q-1\}$ for the irreducible characters of $G_P$, where $\chi_P$ is the fundamental character of $G_P$ obtained by its action on the local uniformizer $u_P$ at $P$ considered modulo $\mathfrak{m}_P^2=\langle u_P^2\rangle$. Note that for any $P\in\pi^{-1}(Q)$, one has that $e_Q>1$, i.e. $G_P$ is non-trivial, if and only if $P\in\mathfrak{r}\subset X$, the ramification locus of $\pi$; equivalently $Q\in\mathfrak{b}\subset Y$, the branch locus of $\pi$.

\begin{definition}\label{def:localmult}
Let $Q\in {Y}$, let $P\in\pi^{-1}(Q)$ and let $\mathcal{E}\in{\rm QCoh}_G(X)$ be a vector bundle.
\begin{enumerate}
\item {The fiber of $\mathcal{E}$ at $P$ is the $kG_P$-module
\[
V_{P}(\mathcal{E})=\mathcal{E}_P\otimes_{\mathcal{O}_{X,P}} k(P),
\]
where $k(P)=k$, as $k$ is algebraically closed.}
\item The local multiplicities of an irreducible $kG$-module $I$ with respect to $\mathcal{E}$ are the integers
\[
n_{d,Q,I}(\mathcal{E})=\langle \chi_P^d, V_P(\mathcal{E})\otimes {\rm Res}^G_{G_P}(I)\rangle,
\text{ where }
0\leq d\leq e_Q-1{,}
\]
{and we extend the first index periodically, $n_{d+e_Q,Q,I}(\mathcal{E})=n_{d,Q,I}(\mathcal{E})$ for all $d\in\mathbb{Z}$; in particular $n_{e_Q,Q,I}(\mathcal{E})=n_{0,Q,I}(\mathcal{E})$. These numbers do not depend on the choice of $P\in\pi^{-1}(Q)$.}
In the special case that $\mathcal{E}=\mathcal{O}_X$, we write $n_{d,Q,I}$ for $n_{d,Q,I}(\mathcal{O}_X)=\langle \chi_P^d, {\rm Res}^G_{G_P}(I)\rangle$.
\end{enumerate}
\end{definition}

\vspace{3mm}

\noindent{\em Remark.} {The fiber is canonically isomorphic to the fiber of the completed stalk, $V_{P}(\mathcal{E})\cong \widehat{\mathcal{E}}_P\otimes_{\widehat{\mathcal{O}}_{X,P}} k$, so completion does not change the representation. This links the definition with the semilocal presentation of \cite{MR589254}: setting $C_Q=(\pi_*\mathcal{O}_X)_Q\otimes_{\mathcal{O}_{Y,Q}}\widehat{\mathcal{O}}_{Y,Q}$ and $\mathcal{E}^\wedge_Q=(\pi_*\mathcal{E})_Q\otimes_{\mathcal{O}_{Y,Q}}\widehat{\mathcal{O}}_{Y,Q}$, by the discussion following \cite[Lemma 3.4]{MR589254} one has $C_Q\cong\prod_{P'\in\pi^{-1}(Q)}\widehat{\mathcal{O}}_{X,P'}$, and $V_P(\mathcal{E})\cong(\mathcal{E}^\wedge_Q\otimes_{C_Q}\widehat{\mathcal{O}}_{X,P})\otimes_{\widehat{\mathcal{O}}_{X,P}}k$. If $Q$ is not a branch point, then $e_Q=1$ and only the index $d=0$ occurs.}

\vspace{1.5mm}

Let ${\rm Rep}_k(G)$ be the representation ring of $G$ over $k$ and let $[W]\in{\rm Rep_k}(G)$ denote the virtual representation corresponding to a $kG$-module $W$.
\begin{definition}\label{def:Euler}
The equivariant Euler characteristic of a $G$-equivariant vector bundle $\mathcal{E} \in {\rm QCoh}_G(X)$ is the virtual representation
\[
\chi_G(\mathcal{E})=[H^0(\mathcal{E})]-[H^1(\mathcal{E})]\in{\rm Rep}_k(G).
\]
\end{definition}

The reader is referred to  {\cite[(3.7)]{MR589254}} for the proof of the following.
\begin{theorem} \label{th:EL}
If $\mathcal{E}\in{\rm QCoh}_G(X)$ is a vector bundle on $X$, then
\[
\chi_G(\mathcal{E})
=
  \left(
\frac{1}{|G|} \deg(\mathcal{E})+(1-g_Y) \mathrm{rk}(\mathcal{E})
  \right)[kG]
  -
  \sum_{I \in \mathrm{Irr}(G)}
  \sum_{Q\in Y}
  \sum_{d=1}^{e_Q}
\left(
1 -\frac{d}{e_Q}
\right)
n_{d,Q,{I^\vee}}(\mathcal{E}) [I],
\]
where $Y=X/G$, $g_Y$ is the genus of $Y$, ${\rm Irr}(G)$ is the set of irreducible representations of $G$, $e_Q$ is the ramification index of $P\in\pi^{-1}(Q)$ and the $n_{d,Q,I}(\mathcal{E})$ are as in Definition \ref{def:localmult}. {(Formula (3.7) of \cite{MR589254} contains a misprint: $n_{d,Q,I}(\mathcal{E})$ there should read $n_{d,Q,I^\vee}(\mathcal{E})$.) The indices $n_{e_Q,Q,I^\vee}$ are understood periodically, as in Definition~\ref{def:localmult}; for unramified $Q$ the weight $1-d/e_Q$ vanishes, so the sum over $Q\in Y$ is finite. The individual terms are elements of ${\rm Rep}_k(G)\otimes_{\mathbb{Z}}\mathbb{Q}$ (their rational coefficients are not elements of $k$), and the total lies in ${\rm Rep}_k(G)$.}
\end{theorem}

Setting $G$ to be the trivial group, the ramification terms vanish and the quotient map becomes the identity, and so the result above specializes to the classical Riemann–Roch theorem for vector bundles on curves. Among its many applications, the latter may be used to calculate Euler characteristics of direct sums, tensor products and exterior powers. The goal being to explore similar techniques in the equivariant case, we first examine the behaviour of $V_P(-)$ with respect to these operations.
\begin{lemma}\label{lem:localprops}
Let $\mathcal{E}, \mathcal{E}',\mathcal{E}''$ be $G$-equivariant vector bundles.
\begin{enumerate}
\item If $0\rightarrow \mathcal{E}'\rightarrow \mathcal{E}\rightarrow\mathcal{E}''\rightarrow 0$ is a short exact sequence in ${\rm QCoh}_G(X)$, then
$
{V_P(\mathcal{E})\cong V_P(\mathcal{E}')\oplus V_P(\mathcal{E}'')}
$ {as $kG_P$-modules}.\\
\item $V_{P}(\mathcal{E}\otimes \mathcal{E}')\cong V_{P}(\mathcal{E})\otimes V_P( \mathcal{E}')$ and $V_P\big(\bigwedge^r\mathcal{E}\big)\cong \bigwedge^r V_P(\mathcal{E})$.\\
\item If {$W\in kG$-mod, then $V_P(W\otimes \mathcal{E})\cong {\rm Res}^G_{G_P}W\otimes V_P(\mathcal{E})$, where $W\otimes\mathcal{E}=(W\otimes_k\mathcal{O}_X)\otimes\mathcal{E}$ carries the diagonal $G$-linearization}.
\end{enumerate}
\end{lemma}
\begin{proof}
Using the isomorphism $V_{P}(\mathcal{E})\cong \widehat{\mathcal{E}}_P\otimes_{\widehat{\mathcal{O}}_{X,P}} \widehat{\mathcal{O}}_{X,P}/\mathfrak{m}_P$ discussed in the Remark following Definition \ref{def:localmult},  we regard $V_{P}(\mathcal{E})$ as the fiber of a completed stalk. Formation of stalks commutes with direct sums, tensor products and exterior powers \cite[7.4.1, 7.4.7 and 7.20.5]{MR4225278}. Completion commutes with direct sums \cite[7.15]{Eisenbud:95}; interpreting it as flat base change, \cite[7.2]{Eisenbud:95}, it also does so with tensor products,  \cite[Appendix A, Formula 11]{MR1011461}, and exterior powers, \cite[Appendix C, (4)]{MR1011461}. The exact same arguments apply to taking fibers, as it amounts to tensoring free modules with residue fields. {For (1), the sequence of fibers $0\rightarrow V_P(\mathcal{E}')\rightarrow V_P(\mathcal{E})\rightarrow V_P(\mathcal{E}'')\rightarrow 0$ is exact, since $\mathcal{E}''$ is locally free; as $|G_P|$ is invertible in $k$, Maschke's theorem shows that this sequence of $kG_P$-modules splits (non-canonically).}
\end{proof}
It follows directly from part (1) of Lemma \ref{lem:localprops} that equivariant Euler characteristics are additive with respect to short exact sequences. While their behavior under general tensor products and exterior powers is governed by more complex combinatorial rules, the specific case of tensoring with trivial bundles is tractable and particularly relevant to the present analysis.
\begin{corollary}\label{eq:prodchar}
Let $\mathcal{E}, \mathcal{E}',\mathcal{E}''$ be $G$-equivariant vector bundles.
\begin{enumerate}
\item If $0\rightarrow \mathcal{E}'\rightarrow \mathcal{E}\rightarrow\mathcal{E}''\rightarrow 0$ is a short exact sequence in ${\rm QCoh}_G(X)$, then
$
\chi_G(\mathcal{E})= \chi_G(\mathcal{E}')+\chi_G(\mathcal{E}'')
$.\\
\item If $W\in kG$-mod, then
\begin{multline*}
\chi_G(W\otimes\mathcal{E})=\dim_k(W)  \left(
\frac{1}{|G|} \deg(\mathcal{E})+(1-g_Y) \mathrm{rk}(\mathcal{E})
  \right)[kG]
\\
  -
  \sum_{I \in \mathrm{Irr}(G)}
  \sum_{Q\in Y}
  \sum_{d=1}^{e_Q}
\sum_{k=0}^{e_Q-1}
\left(
1 -\frac{d}{e_Q}
\right)\langle \chi_P^{d-k}, W\rangle \; n_{k,Q,{I^\vee}}(\mathcal{E}) [I].
\end{multline*}
\end{enumerate}

\end{corollary}
\begin{proof}
(1) follows immediately by part(1) of Lemma \ref{lem:localprops}. For (2), since $\deg(W\otimes\mathcal{E})=\dim_k(W)\deg(\mathcal{E})$ and ${\rm rk}(W\otimes\mathcal{E})=\dim_k(W){\rm rk}(\mathcal{E})$, the coefficient of $k[G]$ in $\chi_G(W\otimes\mathcal{E})$ is $\dim_k(W)$ times the coefficient of $k[G]$ in $\chi_G(\mathcal{E})$. Regarding the triple sum, parts (2) and (3) of Lemma \ref{lem:localprops} imply that
\[
n_{d,Q,I}(W\otimes\mathcal{E})=
\langle \chi_P^d,V_P(W\otimes\mathcal{E})\otimes I\rangle
=
\langle \chi_P^d,W\otimes V_P(\mathcal{E})\otimes I\rangle
=
\sum_{k=0}^{e_Q-1}\langle \chi_P^{d-k}, W\rangle \; n_{k,Q,I}(\mathcal{E}).
\]
{Applying this identity with $I^\vee$ in place of $I$ and using Theorem~\ref{th:EL} yields (2).}
\end{proof}
Ranks, degrees and genera are, more often than not, straightforward to compute; thus, the main challenge to obtaining explicit results from Theorem \ref{th:EL} is to determine its local ingredients. The best that one can hope for is to express each $n_{d,Q,I}(\mathcal{E})$ in terms of $n_{d,Q,I}=n_{d,Q,I}(\mathcal{O}_X)$.
 For $\mathcal{E}=\Omega_X^{\otimes m},\;m\in\mathbb{Z}_{\geq 0}$, one has the following formula, which is
implicit in \cite[3.9]{MR589254}; for clarity of the exposition, we include a short proof.
\begin{lemma}\label{lem:nomega}
For all $Q\in Y,\;1\leq d\leq e_Q,\; I\in{\rm Irr}(G)$ and $m\in\mathbb{N}$,
\[
n_{d,Q,I}\big(\Omega_X^{\otimes m}\big)=n_{e_Q\left\{\frac{m-d}{e_Q}\right\},Q,I^\vee},
\]
where $\{ x\}=x-\lfloor x \rfloor$ denotes the fractional part of $x\in{\mathbb{Q}}$ and $I^\vee$ is the dual of $I$.
\end{lemma}
\begin{proof}
If $P\in \pi^{-1}(Q)$, then $V_P(\Omega_X)\cong \mathfrak{m}_P/\mathfrak{m}_P^2$ as $kG_P$-modules, and so the character of $V_P(\Omega_X)$ is the fundamental character $\chi_P$. Observing that $n_{d,Q,I}=n_{e_Q-d,Q,I^{\vee}}$, one writes
\[
V_P\big(\Omega_X^{\otimes m}\big )\otimes {\rm Res}^G_{G_P}(I)=
\sum_{{a}=1}^{e_Q}n_{{a},Q,I}\chi_P^{{a}+m}=
\sum_{{a}=1}^{e_Q}n_{e_Q-{a},Q,I^\vee}\chi_P^{{a}+m}=
\sum_{{a}=0}^{e_Q-1}n_{{a},Q,I^\vee}\chi_P^{e_Q-{a}+m}=
\sum_{{a}=0}^{e_Q-1}n_{{a},Q,I^\vee}\chi_P^{m-{a}},
\]
{so that the coefficient of $\chi_P^d$ is $n_{a,Q,I^\vee}$ with $a\equiv m-d \bmod e_Q$,}
and the result follows as the remainder of the division of $m-d$ by $e_Q$ is exactly $e_Q\left\{\frac{m-d}{e_Q}\right\}$. It should be noted that the latter takes values between $0$ and $e_Q-1$, even though $1\leq d\leq e_Q$.
\end{proof}

\vspace{3mm}

\subsection{Decomposition of the kernel bundle term}
We proceed with applying the results of the previous section to determine the $kG$-structure of $[H^0\big(\bigwedge^{p}\mathcal{M}_{\Omega_X}\otimes \Omega_X^{\otimes 2}\big)]$. First we show that this can be obtained directly from Theorem (\ref{th:EL}).
\begin{proposition}\label{prop:vanish}
$
H^1\big(\bigwedge^p\mathcal{M}_{\Omega_X}\otimes \Omega_X^{\otimes 2}\big)=0
$ if $p\neq g-1$, and $
H^1\big(\bigwedge^{g-1}\mathcal{M}_{\Omega_X}\otimes \Omega_X^{\otimes 2}\big)\cong {D=\det V}
$ {as $kG$-modules. Consequently $K_{g-2,3}(X,\Omega_X)\cong D$.}
\end{proposition}
\begin{proof}
As in the proof of Theorem \ref{th:main}, the short exact sequence $0\rightarrow\mathcal{M}_{\Omega_X}\rightarrow H^0(\Omega_X)\otimes\mathcal{O}_X\rightarrow \Omega_X\rightarrow 0$ gives rise, for any $p\in\mathbb{Z}_{\geq 0}$, to a short exact sequence
\[
0\rightarrow
\bigwedge^{p}\mathcal{M}_{\Omega_X}\otimes {\Omega_X}^{\otimes 2}
\rightarrow
\bigwedge^{p}V\otimes {\Omega_X}^{\otimes 2}
\rightarrow
\bigwedge^{p-1}\mathcal{M}_{\Omega_X}\otimes {\Omega_X}^{\otimes 3}
\rightarrow 0.
\]
Noting that $H^1(\Omega^{\otimes 2})=0$, the induced long exact sequence in cohomology in ${\rm Vect}_k$ takes the form
\[
0\rightarrow
{H^0\big(\bigwedge^{p}\mathcal{M}_{\Omega_X}\otimes {\Omega_X}^{\otimes 2}\big)}
\rightarrow
\bigwedge^{p}V\otimes H^0\big({\Omega_X}^{\otimes 2}\big)
\rightarrow
H^0\big(\bigwedge^{p-1}\mathcal{M}_{\Omega_X}\otimes {\Omega_X}^{\otimes 3}\big)
\rightarrow
H^1\big(\bigwedge^{p}\mathcal{M}_{\Omega_X}\otimes {\Omega_X}^{\otimes 2}\big)
\rightarrow 0.
\]
{This is a sequence of $kG$-modules, and by the proof of Theorem~\ref{th:main}, $K_{p-1,3}(X,\Omega_X)$ is the cokernel of its middle map; hence $K_{p-1,3}(X,\Omega_X)\cong H^1\big(\bigwedge^{p}\mathcal{M}_{\Omega_X}\otimes {\Omega_X}^{\otimes 2}\big)$ as $kG$-modules. For $p\neq g-1$ the vanishing follows from $K_{p-1,3}(X,\Omega_X)=0$, see the discussion preceding Corollary \ref{cor:bettivanish} (for $p=0$ also directly from $H^1(\Omega_X^{\otimes 2})=0$, and for $p>g-1$ the exterior power is zero). For $p=g-1$, taking determinants in the equivariant evaluation sequence $0\rightarrow\mathcal{M}_{\Omega_X}\rightarrow V\otimes\mathcal{O}_X\rightarrow\Omega_X\rightarrow 0$ gives $\bigwedge^{g-1}\mathcal{M}_{\Omega_X}=\det\mathcal{M}_{\Omega_X}\cong D\otimes\Omega_X^{-1}$, where the one-dimensional representation $D=\det V$ is regarded as a constant equivariant bundle. Hence $\bigwedge^{g-1}\mathcal{M}_{\Omega_X}\otimes\Omega_X^{\otimes 2}\cong D\otimes\Omega_X$ and $H^1(D\otimes\Omega_X)\cong D\otimes H^1(\Omega_X)\cong D$, since by equivariant Serre duality $H^1(\Omega_X)$ is dual to $H^0(\mathcal{O}_X)=k$, on which $G$ acts trivially. The determinant factor cannot be omitted unless $G$ acts on $V$ with determinant one.}
\end{proof}
Next, we compute the rank and the degree of $\bigwedge^{p}\mathcal{M}_{\Omega_X}\otimes {\Omega_X}^{\otimes 2}$.
\begin{proposition}\label{prop:rkdeg}
\[
{\rm rk}\big(\bigwedge^p\mathcal{M}_{\Omega_X}\otimes \Omega_X^{\otimes 2}\big)
=
\binom{g-1}{p}
\text{ and }
{\deg\big(\bigwedge^p\mathcal{M}_{\Omega_X}\otimes \Omega_X^{\otimes 2}\big)}=
2(2g-2-p)\binom{g-1}{p}.
\]
\end{proposition}
\begin{proof}
Both ${\rm rk}$ and ${\rm deg}$ are additive in short exact sequences. As ${\rm rk}\big(\Omega_X\big)=1$ and ${\rm deg}\big(\Omega_X\big)=2g-2$, $0\rightarrow \mathcal{M}_{\Omega_X}\rightarrow H^0(\Omega_X)\otimes\mathcal{O}_X\rightarrow {\Omega}_X\rightarrow 0$ gives ${\rm rk}\big(\mathcal{M}_{\Omega_X}\big)=g-1\text{ and }{\rm deg}\big(\mathcal{M}_{\Omega_X}\big)=-2g+2$. Thus
\[
{\rm rk}\big(\bigwedge^p\mathcal{M}_{\Omega_X}\big)=\binom{{\rm rk}(\mathcal{M}_{\Omega_X})}{p}
=\binom{g-1}{p}
,\;
{\rm deg}\big(\bigwedge^p\mathcal{M}_{\Omega_X}\big)
=
\binom{{\rm rk}(\mathcal{M}_{\Omega_X})-1}{p-1}{\rm deg}\big(\mathcal{M}_{\Omega_X}\big)
=
\binom{g-2}{p-1}(2-2g).
\]
Finally,
\begin{align*}
{\rm rk}\big(\bigwedge^p\mathcal{M}_{\Omega_X}\otimes \Omega_X^{\otimes 2}\big)
&=
{\rm rk}\big(\bigwedge^p\mathcal{M}_{\Omega_X}\big)\cdot{\rm rk}\big(\Omega_X^{\otimes 2}\big)=\binom{g-1}{p}\text{ and }\\
{\rm deg}\big(\bigwedge^p\mathcal{M}_{\Omega_X}\otimes \Omega_X^{\otimes 2}\big)
&=
{\rm rk}\big(\bigwedge^p\mathcal{M}_{\Omega_X}\big)\cdot{\rm deg}\big(\Omega_X^{\otimes 2}\big)
+
{\rm deg}\big(\bigwedge^p\mathcal{M}_{\Omega_X}\big)\cdot{\rm rk}\big(\Omega_X^{\otimes 2}\big)
\\
&=
(4g-4)\binom{g-1}{p}-(2g-2)\binom{g-2}{p-1}=
4(g-1)\binom{g-1}{p}-2p\binom{g-1}{p}.
\end{align*}
\end{proof}

It remains to compute the $n_{d,Q,I}(\bigwedge^p\mathcal{M}_{\Omega_X}\otimes \Omega_X^{\otimes 2})$; by Lemma \ref{lem:nomega}, it suffices to express them in terms of the $n_{d,Q,I}(\Omega_X^{\otimes m})$ for some $m\in\mathbb{Z}_{\geq 0}$. With this objective in mind, recall from \cite[23.39]{MR1153249} that the representation ring of $G_P$ over $k$ is endowed with a $\lambda$-ring structure, with $\lambda^n(W)=\bigwedge^nW$. If $\lambda_t(V)={\sum_n} \bigwedge^n V t^n$ denotes the respective generating function and $0\rightarrow W_1\rightarrow W_2\rightarrow W_3\rightarrow 0$ is a short exact sequence of $kG_P$-modules then by \cite[12.5]{MR1038525} one has that
\[
\lambda_t(W_2)=\lambda_t({W_1})\lambda_t(W_3)\text{ and }
\left(\lambda_t(W_3)\right)^{-1}=\sigma_{-t}(W_3):=\sum_n(-1)^n{\rm Sym}^n(W_3)t^n.
\]
This has two implications that are of interest to this paper. First, that expanding $\lambda_{t}(W_1)=\lambda_t(W_2)\sigma_{-t}(W_3)$ and equating coefficients yields the following identity of virtual representations
\begin{equation}\label{eq:binomial1}
\bigwedge^p [W_1]= \sum_{i=0}^p (-1)^{i}\bigwedge^{p-i}[W_2]\cdot [{\rm Sym}^i(W_3)].
\end{equation}
Second, that for non-negative integers $a_0,\ldots,a_{e_Q-1}$,
\[
\lambda_t\left(\sum_{d=0}^{e_Q-1}a_d\chi_P^d\right)
=\prod_{d=0}^{e_Q-1}(1+\chi_P^d t)^{a_d}
=\sum_{0\leq k_0,\ldots k_{e_Q-1}\leq a_d}
\left(\prod_{d=0}^{e_Q-1} \binom{a_d}{k_d}\right) \chi^{\sum d\cdot k_d} t^{\sum k_d},
\]
and thus, extracting the coefficient of $t^r$,
\begin{equation}\label{eq:multinomial}
\bigwedge^{r}\left(\sum_{d=0}^{e_Q-1}a_d\chi_P^d\right)
=\sum_{\substack{k_0+\cdots+k_{e_Q-1}=r \\0\leq k_d\leq a_d}}
\left(\prod_{d=0}^{e_Q-1} \binom{a_d}{k_d}\right) \chi^{\sum d\cdot k_d}.
\end{equation}

\begin{proposition}\label{prop:binomial3}
If $Q\in Y,\;0\leq d\leq e_Q-1,\;I\in{\rm Irr}(G)$ and $p\in\mathbb{Z}_{\geq 0}$, then
\[
n_{d,Q,I}\big(\bigwedge^p\mathcal{M}_{\Omega_X}\otimes \Omega_X^{\otimes 2}\big)
=\sum_{i=0}^p
\sum_{k=0}^{e_Q-1}
(-1)^i
\langle
\chi_P^{d-k},
\bigwedge^{p-i}H^0(\Omega_X)
\rangle
\;
n_{k,Q,I}(\Omega_X^{\otimes (i+2)}).
\]
\end{proposition}
\begin{proof}
Part (1) of Lemma \ref{lem:localprops} applied to the short exact sequence $0\rightarrow \mathcal{M}_{\Omega_X}\rightarrow H^0(\Omega_X)\otimes\mathcal{O}_X\rightarrow {\Omega}_X\rightarrow 0$ in ${\rm QCoh}_G(X)$ gives rise to the short exact sequence $0\rightarrow V_P(\mathcal{M}_{\Omega_X})\rightarrow H^0(\Omega_X)\rightarrow V_{P}(\Omega_X)\rightarrow 0$ of $kG_P$-modules. As $V_{P}(\Omega_X)$ is 1-dimensional,  ${\rm Sym}^i(V_{P}(\Omega_X))\cong V_{P}(\Omega_X)^{\otimes i}$ and so eq. (\ref{eq:binomial1}) gives
\[
\bigwedge^p [V_P(\mathcal{M}_{\Omega_X})]
=
\sum_{i=0}^p (-1)^i\bigwedge^{p-i}[H^0(\Omega_X)]\cdot[ V_P(\Omega_X)]^{i}.
\]
Combining this with part (2) of Lemma \ref{lem:localprops} yields
\[
\big[V_P\big(\bigwedge^p\mathcal{M}_{\Omega_X}\otimes \Omega_X^{\otimes 2}\big)\big]
=
\bigwedge^p [V_P(\mathcal{M}_{\Omega_X})]\cdot [V_P(\Omega_X)]^{ 2}
=
\sum_{i=0}^p (-1)^i \bigwedge^{p-i}[H^0(\Omega_X)]\cdot [ V_P(\Omega_X)]^{ i+2},
\]
and thus, as in the discussion preceding Lemma \ref{lem:nomega},
\begin{align*}
n_{d,Q,I}\big(\bigwedge^p\mathcal{M}_{\Omega_X}\otimes \Omega_X^{\otimes 2}\big)
&=
\sum_{i=0}^p(-1)^i
\langle\chi_P^d,\bigwedge^{p-i}H^0(\Omega_X)\otimes V_P(\Omega_X^{\otimes (i+2)})\otimes I
\rangle
\\
&=\sum_{i=0}^p
\sum_{k=0}^{e_Q-1}
(-1)^i
\langle
 \chi_P^{d-k},
\bigwedge^{p-i}H^0(\Omega_X)
\rangle
\;
\langle \chi_P^k,
 V_P(\Omega_X^{\otimes (i+2)})\otimes I
\rangle\\
&=\sum_{i=0}^p
\sum_{k=0}^{e_Q-1}
(-1)^i
\langle
\chi_P^{d-k},
\bigwedge^{p-i}H^0(\Omega_X)
\rangle
\;
n_{k,Q,I}(\Omega_X^{\otimes (i+2)}).
\end{align*}

\end{proof}

To give an explicit formula for the local multiplicities of $\bigwedge^{p-i}H^0(\Omega_X)$, we recall that by \cite[3.9]{MR589254}
\begin{equation}\label{eq:ChW}
[H^0\big(\Omega_X\big)]=\sum_{I\in{\rm Irr}(G)} m_I[I],
\text{ where }
m_I=(g_Y-1)\dim I+\sum_{Q\in Y}\sum_{d=0}^{e_Q-1}{\left\{ \frac{-d}{e_Q}\right\}} n_{d,Q,I}+\delta_{I,I_{\rm triv}}
\end{equation}
where $\delta_{I,I_{\rm triv}}=1$ if $I=I_{\rm triv}$ and $0$ otherwise.
\begin{lemma}\label{lem:localwedge}
For $Q\in Y,\;0\leq d\leq e_Q-1,\;r\in\mathbb{Z}_{\geq 0}$ and $\{m_I:\;I\in {\rm Irr}(G)\}$ as in eq. (\ref{eq:ChW}),
\[
\langle
\chi_P^{d}, \bigwedge^{r}H^0(\Omega_X)
\rangle
=
\sum_{\substack{ 0\leq k_j\leq \sum_{I}m_I n_{j,Q,I}  \\ \sum k_j=r \\ \sum jk_j\equiv d\mod e_Q}}
\prod_{j=0}^{e_Q-1}
\left(\!\begin{array}{c} \displaystyle\sum_{I\in{\rm Irr}(G)}m_I\cdot n_{j,Q,I} \\ k_j \end{array}\!\right).
\]
\end{lemma}
\begin{proof}
The restriction functor commutes with direct sums and exterior powers, and so
\begin{align*}
[{\rm Res}^G_{G_P}\big(\bigwedge^{r}H^0(\Omega_X)\big)]
&=
\bigwedge^{r}\sum_{I\in{\rm Irr}(G)} m_I{\rm Res}^G_{G_P}\big([I]\big)
=
\bigwedge^{r}\sum_{I\in{\rm Irr}(G)} m_I\sum_{d=0}^{e_Q-1}n_{d,Q,I}\chi_P^d\\
&=
\bigwedge^{r}\sum_{d=0}^{e_Q-1}\left(\sum_{I\in{\rm Irr}(G)} m_I\cdot n_{d,Q,I}\right)\chi_P^d.
\end{align*}
The result follows by imposing the relation $\sum_{j}j k_j\equiv d\;{\rm mod}\; e_Q$ to eq. (\ref{eq:multinomial}) to extract the coefficient of $\chi_P^d$.
\end{proof}

For the reader’s convenience, and since the formula below is the main result of the section, we recall within the statement the notation for all of its constituent components.

\begin{theorem}\label{th:mainwedge}
Let $g_X$ be the genus of $X$ and $g_Y$ the genus of $Y=X/G$. For each $Q\in Y$, let $e_Q$ be the ramification index of $P\in\pi^{-1}(Q)$, let $G_P$ be its {stabilizer}, $\chi_P$ be the fundamental character of $\mathfrak{m}_P/\mathfrak{m}_P^2$ and let $n_{k,Q,I}=\langle \chi_P^k,I\rangle$, for $I\in{\rm Irr}(G), \;0\leq k\leq e_Q-1$. Then{, for $0\leq p\leq g_X-2$,}
\begin{multline*}
[H^0\big(\bigwedge^p\mathcal{M}_{\Omega_X}\otimes\Omega_X^{\otimes 2}\big)]
=
\binom{g_X-1}{p} \left( \frac{2(g_X-1-p)}{|G|}
+
g_Y-1\right)
[kG]\\
  +
\sum_{I \in \operatorname{Irr}(G)} \sum_{Q \in Y} \sum_{i=0}^p (-1)^i \sum_{d=0}^{e_Q-1}
\langle \chi_P^{d}, \bigwedge^{p-i}H^0(\Omega_X) \rangle
\sum_{k=0}^{e_Q-1} \left\{ \frac{d+i+1-k}{e_Q} \right\} n_{k,Q,I} [I].
\end{multline*}
\end{theorem}
\begin{proof}
{Since $0\leq p\leq g_X-2$, Proposition~\ref{prop:vanish} gives $H^1\big(\bigwedge^p\mathcal{M}_{\Omega_X}\otimes\Omega_X^{\otimes 2}\big)=0$, so the left-hand side equals the equivariant Euler characteristic.} Combining Theorem (\ref{th:EL}) and Proposition \ref{prop:vanish}, gives

\begin{align*}
[H^0\big(\bigwedge^p\mathcal{M}_{\Omega_X}\otimes\Omega_X^{\otimes 2}\big)]
=&
  \left(
\frac{1}{|G|} \deg\big(\bigwedge^p\mathcal{M}_{\Omega_X}\otimes\Omega_X^{\otimes 2}\big)+(1-g_Y) \mathrm{rk}\big(\bigwedge^p\mathcal{M}_{\Omega_X}\otimes\Omega_X^{\otimes 2}\big)
  \right)[kG]\\
  &-
  \sum_{I \in \mathrm{Irr}(G)}
  \sum_{Q\in Y}
  \sum_{d=1}^{e_Q}
\left(
1 -\frac{d}{e_Q}
\right)
n_{d,Q,{I^\vee}}\big(\bigwedge^p\mathcal{M}_{\Omega_X}\otimes\Omega_X^{\otimes 2}\big) \cdot [I].
\end{align*}
Substituting the degree and the rank from Proposition \ref{prop:rkdeg} into the coefficient of $[kG]$ gives
\begin{equation}\label{eq:rkdeg}
\binom{g_X-1}{p} \left(
\frac{2(2g_X-2-p)}{|G|} +(1-g_Y)
  \right)[kG].
\end{equation}
Regarding the triple sum, Proposition \ref{prop:binomial3} and Lemma \ref{lem:nomega}{, applied to $I^\vee$ (note that $(I^\vee)^\vee=I$),} imply that
\[
\left(1-\frac{d}{e_Q}\right)n_{d,Q,{I^\vee}}\big(\bigwedge^p\mathcal{M}_{\Omega_X}\otimes \Omega_X^{\otimes 2}\big)
=
\sum_{i=0}^{p}\sum_{k=0}^{e_Q-1}
(-1)^i
\left(1-\frac{d}{e_Q}\right)
\;
\langle
\chi_P^{d-k},
\bigwedge^{p-i}H^0(\Omega_X)
\rangle
\;
n_{e_Q\left\{ \frac{i+2-k}{e_Q}\right\},Q,{I}},
\]
and observing that $k'=e_Q\left\{ \frac{i+2-k}{e_Q}\right\}$ if and only if $k=e_Q\left\{ \frac{i+2-k'}{e_Q}\right\}$ allows us to rewrite the right hand side
\[
\sum_{i=0}^{p}\sum_{k=0}^{e_Q-1}
(-1)^i
\left(1-\frac{d}{e_Q}\right)
\;
\langle
\chi_P^{d-e_Q\left\{ \frac{i+2-k}{e_Q}\right\}},
\bigwedge^{p-i}H^0(\Omega_X)
\rangle
\;
n_{k,Q,{I}}.
\]
Finally, $d-e_Q\left\{ \frac{i+2-k}{e_Q}\right\}\equiv e_Q\left\{ \frac{d-i-2+k}{e_Q}\right\}{\rm mod }\;e_Q$. Setting  $d'
=
e_Q
\left\{
\frac{d-i-2+k}{e_Q}
\right\}
$ forces
$
d=
1+
e_Q
\left\{
\frac{d'+i+1-k}{e_Q}
\right\}$,
to ensure that $1\leq d \leq e_Q$ with $0\leq d'\leq e_Q-1$.
Thus, we arrive at
\begin{multline*}
\sum_{d=1}^{e_Q}\left(1-\frac{d}{e_Q}\right)n_{d,Q,{I^\vee}}\big(\bigwedge^p\mathcal{M}_{\Omega_X}\otimes \Omega_X^{\otimes 2}\big)
=
\\
\sum_{d=0}^{e_Q-1}\sum_{i=0}^{p}\sum_{k=0}^{e_Q-1}
(-1)^i
\left(
1-\frac{1}{e_Q}-
\left\{
\frac{d+i+1-k}{e_Q}
\right\}
\right)
\;
\langle
\chi_P^{d},
\bigwedge^{p-i}H^0(\Omega_X)
\rangle
\;
n_{k,Q,{I}}.
\end{multline*}
The term $(1-\frac{1}{e_Q})$ is independent of the summation indices, and so we write the RHS as
\begin{multline}\label{eq:crazysum}
\left(1-\frac{1}{e_Q}\right)
\sum_{i=0}^{p}
(-1)^i
\sum_{d=0}^{e_Q-1}
\langle
\chi_P^{d},
\bigwedge^{p-i}H^0(\Omega_X)
\rangle
\sum_{k=0}^{e_Q-1}
n_{k,Q,{I}}
\\
-
\sum_{d=0}^{e_Q-1}\sum_{i=0}^{p}\sum_{k=0}^{e_Q-1}
(-1)^i
\left\{
\frac{d+i+1-k}{e_Q}
\right\}
\;
\langle
\chi_P^{d},
\bigwedge^{p-i}H^0(\Omega_X)
\rangle
\;
n_{k,Q,{I}}.
\end{multline}
Next, we observe that $\sum_k n_{k,Q,{I}}=\dim {I}$ and
\[
\sum_{i=0}^{p}
(-1)^i
\sum_{d=0}^{e_Q-1}
\langle
\chi_P^{d},
\bigwedge^{p-i}H^0(\Omega_X)
\rangle
=
\sum_{i=0}^{p}
(-1)^i \binom{g_X}{p-i}
=
\binom{g_X-1}{p}
\]
and so
\[
\sum_{Q\in Y}\sum_{I\in{\rm Irr}(G)} \left(1-\frac{1}{e_Q}\right)
\sum_{i=0}^{p}
(-1)^i
\sum_{d=0}^{e_Q-1}
\langle
\chi_P^{d},
\bigwedge^{p-i}H^0(\Omega_X)
\rangle
\sum_{k=0}^{e_Q-1}
n_{k,Q,{I}}
[I]
=
\sum_{Q\in Y}\left(1-\frac{1}{e_Q}\right)
\binom{g_X-1}{p}
[kG].
\]
Using the Riemann-Hurwitz formula, the RHS becomes
\[
\binom{g_X-1}{p}\left(\frac{2(g_X-1)}{|G|}-2(g_Y-1)\right)[kG]
\]
and combining this with eq. (\ref{eq:rkdeg}) gives the final coefficient of $[kG]$. The remaining term is obtained from the second line of eq. (\ref{eq:crazysum}). The reader should keep in mind that an explicit formula for the integers $\langle
\chi_P^{d},
\bigwedge^{p-i}H^0(\Omega_X)
\rangle$ has been given in Lemma \ref{lem:localwedge}.
\end{proof}

{\noindent{\em Remark.} The restriction $p\leq g_X-2$ is necessary: already for the trivial group and $p=g_X-1$ the left-hand side is $H^0(\Omega_X)$, of dimension $g_X$, whereas the right-hand side has dimension $g_X-1$. If $R_p$ denotes the right-hand side of Theorem~\ref{th:mainwedge}, then for $0\leq p\leq g_X-1$ one has $[H^0(\bigwedge^p\mathcal{M}_{\Omega_X}\otimes\Omega_X^{\otimes 2})]=R_p+\delta_{p,g_X-1}[D]$, by Proposition~\ref{prop:vanish}.}

\subsection{Schur functors and generating functions}
Let $V=H^0(\Omega_X)$. In this section we study the virtual representation $[\bigwedge^{p}V\otimes V]-[\bigwedge^{p+1}V]$ that appears in Corollary \ref{cor:bettivanish}. This can be realized as the kernel of the canonical epimorphism
\[
\bigwedge^{p}V\otimes V\rightarrow \bigwedge^{p+1}V,\; v_1\wedge\cdots \wedge v_p\otimes v\mapsto v_1\wedge\cdots \wedge v_p\wedge v,
\]
{We define
\[
\mathbb{S}_p(V):=\ker\Big(\bigwedge^{p}V\otimes V\rightarrow \bigwedge^{p+1}V\Big),\qquad p\geq 0,
\]
so that $\mathbb{S}_0(V)=0$ and, the map being surjective, $[\mathbb{S}_p(V)]=[\bigwedge^{p}V\otimes V]-[\bigwedge^{p+1}V]$ in every characteristic. In characteristic zero, see for example \cite[p.79]{MR1153249}, $\mathbb{S}_p(V)$ is, for $p\geq 1$, the {\em Schur functor} $\mathbb{S}_{(2,1^{p-1})}(V)$ associated to the hook partition $(2,1^{p-1})=(2,1,\ldots,1)$ of $p+1$; for instance, $\mathbb{S}_1(V)={\rm Sym}^2V$. The kernel definition suffices for all identities used below.}\\

{To organize the decomposition of $\mathbb{S}_p(V)$ uniformly for finite-group representations, we use generating functions and character-theoretic multiplicity formulas.} In what follows, we present the two classical equivalent techniques.

For each $g\in G$, the generating polynomials
\[
{\rm Sch}_V(t,g)=\sum_{p=0}^{\dim V}\chi_{\mathbb{S}_p(V)}(g)t^p \in k[t]
\]
are given by the result below.

\begin{theorem}\label{th:genchars}
For $g\in G$ and $I\in{\rm Irr}(G)$, let $\xi_{I,1},\ldots,\xi_{I,\dim I}$ be the eigenvalues of the matrix of the action of $g$ on $I$ and let
\[
m_I=(g_Y-1)\dim I+\sum_{Q\in Y}\sum_{d=0}^{e_Q-1}{\left\{ \frac{-d}{e_Q}\right\}} n_{d,Q,I}+\delta_{I,I_{\rm triv}}
\]
be the multiplicity of $I$ in the decomposition of $V$ as in eq. (\ref{eq:ChW}). Then the character value $\chi_{\mathbb{S}_p(V)}(g)$ is the coefficient of $t^p$ in the expression below.
\[
\frac{1}{t}
\left[
\left(
\sum_{I}m_I\cdot\chi_I(g)t-1
\right)
\prod_{I\in{\rm Irr}(G)}\prod_{j=1}^{\dim I}(1+t\cdot\xi_{I,j}(g))^{m_I}
+1
\right] \in k[t].
\]
\end{theorem}
\begin{proof}
For each $g\in G$, let
\[
{\rm Wd}_V(t,g)=\sum_{p=0}^{\dim V}\chi_{\bigwedge^p V}(g)t^p \in k[t]
\]
and observe that
\[
\sum_{p=0}^{\dim V}\chi_{\bigwedge^{p+1} V}(g)t^p=
\frac{1}{t}\left({\rm Wd}_V(t,g)-1\right).
\]
By definition of $\mathbb{S}_p(V)$, one has that  $\chi_{\mathbb{S}_p(V)}(g)=\chi_{\bigwedge^p V}(g)\chi_V(g)-\chi_{\bigwedge^{p+1} V}(g)$, and so
\begin{equation}\label{eq:Schurid}
{\rm Sch}_V(t,g)
=
\chi_V(g){\rm Wd}_V(t,g)-\frac{1}{t}\left({\rm Wd}_V(t,g)-1\right)
=
\frac{1}{t}
\left[
(
\chi_V(g)t-1
)
{\rm Wd}_V(t,g)
+1
\right].
\end{equation}
If $m_I$ is the multiplicity of $I\in{\rm Irr}(G)$ in the decomposition of $V$, see eq. (\ref{eq:ChW}), then
$
\chi_V(g)=\sum_{I}m_I\chi_I(g),
$
so
\[
{\rm Wd}_V(t,g)=\prod_{I\in{\rm Irr}(G)} {\rm Wd}_I(t,g)^{m_I}.
\]
Let $\rho_I(g)$ be the matrix of the action of $g$ on $I$, let  $\{ v_1,\ldots,v_{\dim I}\}$ be a basis of eigenvectors and let $\xi_{I,1},\ldots,\xi_{I,\dim I}$ be the corresponding eigenvalues. Then $g\in G$ acts on a basis vector of $\bigwedge^p I$ by
\[
g(v_{j_1}\wedge\cdots\wedge v_{j_p})=\xi_{I,j_1}\cdots\xi_{I,j_p}(v_{j_1}\wedge\cdots\wedge v_{j_p}),
\]
and so
\[
\chi_{\bigwedge^p I}(g)=\sum_{ j_1{<} \cdots {<} j_p}\xi_{I,j_1}\cdots\xi_{I,j_p}=e_p(\xi_{I,1},\ldots,\xi_{I,\dim I}),
\]
where $e_p$ denotes the $p$-th elementary symmetric polynomial. Thus
\[
{\rm Wd}_I(t,g)=\det ({\rm Id}+t\rho_I(g))=\prod_{j=1}^{\dim I}(1+t\cdot\xi_{I,j}(g)),
\]
and substituting
\[
{\rm Wd}_V(t,g)=\prod_{I\in{\rm Irr}(G)}\prod_{j=1}^{\dim I}(1+t\cdot\xi_{I,j}(g))^{m_I}
\]
into eq. (\ref{eq:Schurid}) gives the final result.
\end{proof}

One may use the above formula to obtain the character table of $\mathbb{S}_p(V)$ by extracting the coefficient of $t^p$ from each ${\rm Sch}_V(t,g),\;g\in G$. {Theorem~\ref{th:genchars} is an identity of $k$-valued traces and holds whenever ${\rm char}(k)\nmid |G|$. However, if ${\rm char}(k)=\ell>0$, $k$-valued traces only determine multiplicities modulo $\ell$ (e.g. $k^{\oplus\ell}$ and $0$ have the same character). For the multiplicity formulas below we therefore assume that ${\rm char}(k)=0$, and write $\chi_{I^\vee}(g)=\chi_I(g^{-1})$. In positive characteristic $\ell\nmid|G|$ the same formulas hold verbatim if every character is replaced by the corresponding Brauer character, obtained by lifting each eigenvalue (not the trace) to a root of unity in characteristic zero; since every $kG$-module is projective, Brauer-character orthogonality recovers integer multiplicities, see \cite[Thm.~9.4.11, \S10.1 and Prop.~10.2.1]{Webb2016}. The identity for $\sigma_V(t)$ below holds integrally in ${\rm Rep}_k(G)[t]$ in all tame characteristics.} An alternative, albeit equivalent viewpoint, is to aim at the multiplicity of each irreducible $I$ in the decomposition of $\mathbb{S}_p(V)$ as follows: instead of considering generating functions in $k[t]$ for each $g\in G$, we switch to generating functions in ${\rm Rep}_k(G)[t]$
\[
\sigma_V(t)=\sum_{p=0}^{\dim V}[\mathbb{S}_p(V)]t^p\text{ and }\lambda_V(t)=\sum_{p=0}^{\dim V} \big[\bigwedge^pV] t^p.
\]
They satisfy the identity, analogously to eq. (\ref{eq:Schurid}),
\[
\sigma_V(t)
=\frac{1}{t}\left((t[V]-[I_{\rm triv}])\lambda_V(t)+[I_{\rm triv}]\right).
\]
Substituting ${s_{p,I}}$, ${e_{p,I}}$ and $m_I$ for the multiplicities of $I\in{\rm Irr}(G)$, $\mathbb{S}_p(V)$, $\bigwedge^p(V)$, and $V$, respectively
\[
\sigma_V(t)
=\sum_{p=0}^{\dim V}\sum_{I\in {\rm Irr}(G)} s_{p,I}[I]t^p
=\sum_{I\in {\rm Irr}(G)}\left(\sum_{p=0}^{\dim V} s_{p,I}t^p\right)[I]
\text{ and }
\lambda_V(t)
=\sum_{I\in {\rm Irr}(G)}\left(\sum_{p=0}^{\dim V} e_{p,I}t^p\right)[I],
\]
allows one to employ the following variant of Molien's theorem for exterior algebras \cite[5.2]{BensonPolInv}
\[
\sum_{p=0}^{\dim V} e_{p,I}t^p
=
\frac{1}{|G|}\sum_{g\in G}
\det({\rm Id}+t\rho_V(g))\chi_{I^\vee}(g)
=
\frac{1}{|G|}\sum_{g\in G}
\prod_{J\in{\rm Irr}(G)}
\det({\rm Id}+t\rho_J(g))^{m_J}{\chi_{I}(g^{-1})},
\]
and so one obtains the following reformulation of Theorem \ref{th:genchars}.
\begin{theorem}\label{th:Molien}
For $g\in G$ and $I\in{\rm Irr}(G)$, let $\xi_{I,1},\ldots,\xi_{I,\dim I}$ be the eigenvalues of the matrix of the action of $g$ on $I$ and let
\[
m_I=(g_Y-1)\dim I+\sum_{Q\in Y}\sum_{d=0}^{e_Q-1}{\left\{ \frac{-d}{e_Q}\right\}} n_{d,Q,I}+\delta_{I,I_{\rm triv}}
\]
be the multiplicity of $I$ in the decomposition of $V$ as in eq. (\ref{eq:ChW}). Then the virtual representation $[\mathbb{S}_p(V)]$  is the coefficient of $t^p$ in the expression below.
\[
\frac{1}{t}
\left(
\frac{1}{|G|}
(t\sum_{I}m_I[I]-[I_{\rm triv}])\sum_{I}\sum_{g\in G}
\prod_{J}
\prod_{j=1}^{\dim J}(1+t\cdot\xi_{J,j}(g))^{m_J}{\chi_{I}(g^{-1})}[I]+[I_{\rm triv}]\right) \in {\rm Rep}_k(G){[t]}.
\]
\end{theorem}

It is worth pointing out, that {the formula for $\sigma_V(t)$ of Theorem \ref{th:Molien} can be directly obtained from the formula for ${\rm Sch}_V(t,g)$ of Theorem \ref{th:genchars}} by applying character orthogonality relations:
\[
\sigma_V(t)
=
\frac{1}{|G|}
\sum_{I\in {\rm Irr}(G)}
\sum_{g\in G}
{\rm Sch}_V(t,g){\chi_I(g^{-1})}[I].
\]

\section{Totally ramified trigonal Kummer curves}\label{sec:Kummer}

\subsection{Setup}\label{sec:Kummer-setup}
Let $s$ be a positive integer divisible by $3$ {with $s\geq 9$ (so that $g=s-2\geq 7$, as required by the results used below)}, let $a_1,\ldots,a_s$ be distinct elements of $\mathbb{C}$ and let
\begin{equation}\label{eq:trigonal}
f(x,y)=    y^3 - \prod_{i=1}^s (x - a_i).
\end{equation}
Define $X$ to be the unique smooth, projective, irreducible algebraic curve over $\mathbb{C}$ with function field
\[
K(X)\cong \text{Frac}\left( \mathbb{C}[x,y] / \langle f(x,y)\rangle \right).
\]
The finite extension $K(X)/\mathbb{C}(x)$ gives rise to a branched Galois covering map $X\rightarrow \mathbb{P}^1_\mathbb{C}$ of degree 3 with Galois group $G = \mathbb{Z}/3\mathbb{Z}$; thus $X$ is a trigonal curve of genus
\[
g=\frac{(3-1)(s-2)}{2}=s-2.
\]
Let $\zeta_3$ be a primitive $3$-rd root of unity. The cyclic group $G = \mathbb{Z}/3\mathbb{Z}$ acts on $X$ via $(x,y)\mapsto (x,\zeta_3y)$, with quotient $X/G$  isomorphic to $\mathbb{P}^1$ and quotient map the covering map $\pi: X \to \mathbb{P}^1$ discussed above. In what follows, we apply Corollary \ref{cor:bettivanish}, Theorem \ref{th:mainwedge} and Theorem \ref{th:Molien} to get explicit values for the differences $[K_{p,1}(X)]-[K_{p-1,2}(X)]$, then we present an alternative, independent method that yields the same results.

\subsection{Explicit computations}\label{sec:Kummer-explicit}
Let $X$ be the trigonal curve introduced in the previous subsection. As before, let $\mathcal{M}_{\Omega_X}$ be the kernel bundle associated to the canonical sheaf $\Omega_X$, and let $\mathbb{S}_p(V)$ be the Schur functor associated to $V=H^0(\Omega_X)$.
\begin{proposition}
In the above notation one has for $1\leq p \leq g-2$ that
\[
[\mathbb{S}_p(V)] \cong
\bigoplus_{j=0}^{p+1} \Bigg[
(p-j) \binom{\frac{s}{3}-1}{j} \binom{\frac{2s}{3}}{p+1-j}
+ \left(\frac{s}{3}-1\right) \binom{\frac{s}{3}-1}{j-1} \binom{\frac{2s}{3}-1}{p+1-j}
\Bigg] \chi_{(p+j+1) \;{\rm mod}\; 3}
\]
and
\begin{multline*}
[H^0\big(\bigwedge^p\mathcal{M}_{\Omega_X}\otimes\Omega_X^{\otimes 2}\big)] =
\left[ \frac{1}{3} \binom{s-3}{p} (2s - 2p - 9) \right] \sum_{r=0}^2\chi_r \\
+ \sum_{i=0}^p \sum_{r=0}^2 \sum_{d=0}^{2} \sum_{\substack{0 \le j \le p-i \\ (p-i+j) \equiv d \;{\rm mod}\; 3}}  (-1)^i s \binom{2s/3 - 1}{p-i-j} \binom{s/3 - 1}{j}
\left\{ \frac{d+i+1-r}{3} \right\} \cdot \chi_r.
\end{multline*}
\end{proposition}
\begin{proof}
The trigonal curve $X$ is in particular a Kummer curve and it is well-known, see for example \cite[Theorem 8]{KaranProc}, that $V = H^0(\Omega_X)$ has $k$-basis given by
\begin{equation}\label{eq:basis}
\left\{ x^k \frac{dx}{y} \;\middle|\; 0 \le k \le \frac{s}{3} - 2 \right\}
\bigcup
 \left\{ x^k \frac{dx}{y^2} \;\middle|\; 0 \le k \le \frac{2s}{3} - 2 \right\}.
\end{equation}
If $\chi_0,\chi_1,\chi_2$ are the irreducible characters of $G$, one verifies that $V$ decomposes as
\[
V \cong \left(\frac{2s}{3} - 1\right) \chi_1 \oplus \left(\frac{s}{3} - 1\right) \chi_2
\]
and computes $[\bigwedge^p V]$ using the identities
\begin{equation}\label{eq:identities}
\bigwedge^p (A \oplus B) \cong \bigoplus_{j=0}^p \bigwedge^{p-j} A \otimes \bigwedge^j B,
\quad
\bigwedge^k d\cdot \chi\cong \binom{d}{k}\chi_{k\;{\rm mod}\;3}
\quad
\text{and }
\chi_\alpha\otimes\chi_\beta=\chi_{(\alpha+\beta)\;{\rm mod}\;3}.
\end{equation}
Then
\begin{align*}
\bigwedge^p V&=
\bigwedge^p
\left(
 \left(\frac{2s}{3} - 1\right) \chi_1 \oplus \left(\frac{s}{3} - 1\right) \chi_2
\right) =
\bigoplus_{j=0}^p
\bigwedge^{p-j} \left(\frac{2s}{3} - 1\right) \chi_1
\otimes
\bigwedge^{j} \left(\frac{s}{3} - 1\right) \chi_2\\
&=
\bigoplus_{j=0}^p
\binom{2s/3-1}{p-j}\chi_1^{p-j}\otimes \binom{s/3-1}{j}\chi_2^{j}=
\bigoplus_{j=0}^p
\binom{2s/3-1}{p-j}\binom{s/3-1}{j}\chi_{(p+j)\;{\rm mod}\;3},
\end{align*}
and thus
\begin{align*}
{[\bigwedge^p V \otimes V]} &=
\bigoplus_{j=0}^p \left(\frac{2s}{3}-1\right) \binom{\frac{s}{3}-1}{j} \binom{\frac{2s}{3}-1}{p-j} \chi_{(p+j+1) \;{\rm mod}\; 3}
\bigoplus_{j=0}^p \left(\frac{s}{3}-1\right) \binom{\frac{s}{3}-1}{j} \binom{\frac{2s}{3}-1}{p-j}\chi_{(p+j+2) \;{\rm mod}\; 3}\\
\\
&=
\bigoplus_{j=0}^{p+1} \Bigg[
\left(\frac{2s}{3}-1\right) \binom{\frac{s}{3}-1}{j} \binom{\frac{2s}{3}-1}{p-j}
+ \left(\frac{s}{3}-1\right) \binom{\frac{s}{3}-1}{j-1} \binom{\frac{2s}{3}-1}{p+1-j}
\Bigg] \chi_{(p+j+1) \;{\rm mod}\; 3}.
\end{align*}
Subtracting $[\bigwedge^{p+1} V]$ as computed above, one concludes that
\begin{equation*} \label{eq:schur_final}
\boxed{
[\mathbb{S}_p(V)] \cong
\bigoplus_{j=0}^{p+1} \Bigg[
(p-j) \binom{\frac{s}{3}-1}{j} \binom{\frac{2s}{3}}{p+1-j}
+ \left(\frac{s}{3}-1\right) \binom{\frac{s}{3}-1}{j-1} \binom{\frac{2s}{3}-1}{p+1-j}
\Bigg] \chi_{(p+j+1) \;{\rm mod}\; 3}.
}
\end{equation*}
\vspace{3mm}

Now consider the formula of Theorem \ref{th:mainwedge}
\begin{multline*}
[H^0\big(\bigwedge^p\mathcal{M}_{\Omega_X}\otimes\Omega_X^{\otimes 2}\big)]
=
\binom{g_X-1}{p} \left( \frac{2(g_X-1-p)}{|G|}
+
g_Y-1\right)
[kG]\\
  +
\sum_{I \in \operatorname{Irr}(G)} \sum_{Q \in Y} \sum_{i=0}^p (-1)^i \sum_{d=0}^{e_Q-1}
\langle \chi_P^{d}, \bigwedge^{p-i}H^0(\Omega_X) \rangle
\sum_{k=0}^{e_Q-1} \left\{ \frac{d+i+1-k}{e_Q} \right\} n_{k,Q,I} [I].
\end{multline*}
Substituting $g_X=s-{ 2},\;g_Y=0$ and $|G|=3$, the global part of the formula becomes
\[
\binom{s-3}{p} \left( \frac{2(s-3-p)}{3}-1\right)[kG]
=
\frac{1}{3}
\binom{s-3}{p} \left( 2s-2p-9\right)[kG].
\]
For the local part, one notes that the cover $X\rightarrow\mathbb{P}^1_\mathbb{C}$ has $s$ branch points, each of ramification index equal to $3$. The local multiplicities are given by the formula
\[
n_{k,Q,\chi_r}=
\begin{cases}
1,\text{ if }k=r\\\
0,\text{ otherwise},
\end{cases}
\]
and the decomposition of $\bigwedge^p V$ obtained above yields
\[
\langle \chi_P^{d}, \bigwedge^{p-i}H^0(\Omega_X) \rangle  = \sum_{\substack{0 \le j \le p-i \\ (p-i+j) \equiv d \pmod 3}} \binom{2s/3 - 1}{p-i-j} \binom{s/3 - 1}{j}.
\]
Thus, the local part of the formula becomes
\[
s \sum_{i=0}^p (-1)^i \sum_{d=0}^{2} \sum_{\substack{0 \le j \le p-i \\ (p-i+j) \equiv d \pmod 3}} \binom{2s/3 - 1}{p-i-j} \binom{s/3 - 1}{j}
\sum_{r=0}^2
\left\{ \frac{d+i+1-r}{3} \right\} \cdot \chi_r
\]
and so overall

\begin{empheq}[box=\fbox]{multline*}
[H^0\big(\bigwedge^p\mathcal{M}_{\Omega_X}\otimes\Omega_X^{\otimes 2}\big)] =
\left[ \frac{1}{3} \binom{s-3}{p} (2s - 2p - 9) \right] \sum_{r=0}^2\chi_r \\
+ \sum_{i=0}^p \sum_{r=0}^2 \sum_{d=0}^{2} \sum_{\substack{0 \le j \le p-i \\ (p-i+j) \equiv d \;{\rm mod}\; 3}}  (-1)^i s \binom{2s/3 - 1}{p-i-j} \binom{s/3 - 1}{j}
\left\{ \frac{d+i+1-r}{3} \right\} \cdot \chi_r
\end{empheq}

\end{proof}
Next, recall the formula of Corollary \ref{cor:bettivanish}
\[
[K_{p,1}(X,\Omega_X)]-[K_{p-1,2}(X,\Omega_X)]=
[\mathbb{S}_p(V)]
-[H^0\big(\bigwedge^{p-1}\mathcal{M}_{\Omega_X}\otimes \Omega_X^{\otimes 2}\big)].
\]
Putting $s=9$, one obtains a trigonal Kummer curve $X$ of genus $7$ and calculates
\begin{table}[h]
\centering
\renewcommand{\arraystretch}{1.5}
\begin{tabular}{|c|c|c|c|}
\hline
$\mathbf{p}$ & $[\mathbb{S}_p(V)]$ & $[H^0(\bigwedge^{p-1}\mathcal{M}\otimes\Omega^{\otimes 2})]$ & $[K_{p,1}] - [K_{p-1,2}]$ \\ \hline
$\mathbf{1}$
& $10\chi_0 + 3\chi_1 + 15\chi_2$
& $6\chi_0 + 3\chi_1 + 9\chi_2$
& $\mathbf{4\chi_0 + 6\chi_2}$ \\ \hline
$\mathbf{2}$
& $42\chi_0 + 50\chi_1 + 20\chi_2$
& $38\chi_0 + 38\chi_1 + 20\chi_2$
& $\mathbf{4\chi_0 + 12\chi_1}$ \\ \hline
$\mathbf{3}$
& $55\chi_0 + 55\chi_1 + 100\chi_2$
& $55\chi_0 + 64\chi_1 + 91\chi_2$
& $\mathbf{-9\chi_1 + 9\chi_2}$ \\ \hline
$\mathbf{4}$
& $100\chi_0 + 80\chi_1 + 44\chi_2$
& $104\chi_0 + 80\chi_1 + 56\chi_2$
& $\mathbf{-4\chi_0 - 12\chi_2}$ \\ \hline
$\mathbf{5}$
& $25\chi_0 + 50\chi_1 + 65\chi_2$
& $29\chi_0 + 56\chi_1 + 65\chi_2$
& $\mathbf{-4\chi_0 - 6\chi_1}$ \\ \hline
\end{tabular}
\caption{Comparison of Schur Functor and Vector Bundle Cohomology for $s=9$.}\label{tab:s=9}
\end{table}

Similarly, for $s=12$, the table corresponding to the genus $10$ trigonal Kummer curve is
\begin{table}[h]
\centering
\renewcommand{\arraystretch}{1.5}
\begin{tabular}{|c|c|c|c|}
\hline
$\mathbf{p}$ & $[\mathbb{S}_p(V)]$ & $[H^0(\bigwedge^{p-1}\mathcal{M}\otimes\Omega^{\otimes 2})]$ & $[K_{p,1}] - [K_{p-1,2}]$ \\ \hline
$\mathbf{1}$
& $21\chi_0 + 6\chi_1 + 28\chi_2$
& $9\chi_0 + 5\chi_1 + 13\chi_2$
& $\mathbf{12\chi_0 + 1\chi_1 + 15\chi_2}$ \\ \hline
$\mathbf{2}$
& $120\chi_0 + 147\chi_1 + 63\chi_2$
& $87\chi_0 + 87\chi_1 + 51\chi_2$
& $\mathbf{33\chi_0 + 60\chi_1 + 12\chi_2}$ \\ \hline
$\mathbf{3}$
& $273\chi_0 + 273\chi_1 + 444\chi_2$
& $228\chi_0 + 264\chi_1 + 336\chi_2$
& $\mathbf{45\chi_0 + 9\chi_1 + 108\chi_2}$ \\ \hline
$\mathbf{4}$
& $756\chi_0 + 651\chi_1 + 441\chi_2$
& $696\chi_0 + 576\chi_1 + 492\chi_2$
& $\mathbf{60\chi_0 + 75\chi_1 - 51\chi_2}$ \\ \hline
$\mathbf{5}$
& $567\chi_0 + 798\chi_1 + 945\chi_2$
& $642\chi_0 + 858\chi_1 + 894\chi_2$
& $\mathbf{-75\chi_0 - 60\chi_1 + 51\chi_2}$ \\ \hline
$\mathbf{6}$
& $861\chi_0 + 573\chi_1 + 546\chi_2$
& $870\chi_0 + 618\chi_1 + 654\chi_2$
& $\mathbf{-9\chi_0 - 45\chi_1 - 108\chi_2}$ \\ \hline
$\mathbf{7}$
& $252\chi_0 + 483\chi_1 + 420\chi_2$
& $312\chi_0 + 516\chi_1 + 432\chi_2$
& $\mathbf{-60\chi_0 - 33\chi_1 - 12\chi_2}$ \\ \hline
$\mathbf{8}$
& $203\chi_0 + 84\chi_1 + 153\chi_2$
& $204\chi_0 + 96\chi_1 + 168\chi_2$
& $\mathbf{-1\chi_0 - 12\chi_1 - 15\chi_2}$ \\ \hline
\end{tabular}
\caption{Comparison of Schur Functor and Vector Bundle Cohomology for $s=12$.}\label{tab:s=12}
\end{table}

\subsection{Action on the resolution}\label{sec:Kummer-resolution}
It is well-known that by the classical Enriques-Petri Theorem, see for example \cite{Saint-Donat73}, the canonical embedding $X\hookrightarrow\mathbb{P}^{g-1} = \mathbb{P}^{s-3}$ of the trigonal curve $X$ lies on a rational normal scroll $S \subset \mathbb{P}^{s-3}$ of dimension 2. Combining this with the mapping cone lemma, see \cite[6.15]{MR2103875}, Schreyer in \cite{MR849058} described the minimal free resolution of {the ideal $I_C$ of the canonical model $C\subset\mathbb{P}^{g-1}$ of $X$} as a mapping cone involving the Eagon-Northcott complex associated with the scroll $S$.

\begin{theorem}[{\cite[Cor.~4.4 and \S 6.1]{MR849058}}]\label{Schreyer}
Let $C \subset \mathbb{P}^{g-1}$ be a canonical trigonal curve of genus $g \ge 5$, let $\mathcal{L}= \pi^* \cO_{\mathbb{P}^1}(1)$ be the line bundle on $C$ associated to the trigonal linear series $g_3^1$ and let
\[
U = H^0(\mathcal{L}),\quad E = H^0(\Omega_C \otimes \mathcal{L}^{-1})\quad \text{ and }\quad R = \operatorname{Sym}(H^0( \Omega_C)).
\]
The minimal free resolution $F_\bullet$ of the defining ideal $I_C$ of $C \subset \mathbb{P}^{g-1}$ is
\[
F_\bullet:0 \longrightarrow {R(-g-1)} \longrightarrow {\mathbf{L}_{g-3} \oplus \mathbf{Q}_{g-3}} \longrightarrow \dots \longrightarrow  \mathbf{L}_2 \oplus \mathbf{Q}_2 \longrightarrow  \mathbf{L}_1 \oplus \mathbf{Q}_1 \longrightarrow I_C \longrightarrow 0,
\]
where for {$1\leq i \leq g-3$, the linear part $\mathbf{L}_i$ and the quadratic part $\mathbf{Q}_i$} are given by
\[
\mathbf{L}_i= \bigwedge^{i+1} E \otimes \operatorname{Sym}_{i-1}( U) \otimes R(-i-1),\quad
\mathbf{Q}_i= \bigwedge^{i-1} E \otimes \operatorname{Sym}_{g-i-3}( U) \otimes R(-i-2).
\]
{The last free module $R(-g-1)$ has rank one; it belongs neither to the linear nor to the quadratic row.}
\end{theorem}
{Let $\mathcal{L}=\pi^*\mathcal{O}_{\mathbb{P}^1}(1)$, so that $U=H^0(\mathcal{L})$ and $E=H^0(\Omega_X\otimes\mathcal{L}^{-1})$ are $kG$-modules. The free modules of Theorem~\ref{Schreyer} are infinite-dimensional; we write $\mathbf{L}_i=\mathbf{L}^{\rm gen}_i\otimes R(-i-1)$ and $\mathbf{Q}_i=\mathbf{Q}^{\rm gen}_i\otimes R(-i-2)$, with the finite-dimensional generator representations
\[
\mathbf{L}^{\rm gen}_i=\bigwedge^{i+1} E \otimes \operatorname{Sym}_{i-1}( U),\qquad \mathbf{Q}^{\rm gen}_i=\bigwedge^{i-1} E \otimes \operatorname{Sym}_{g-i-3}( U),\qquad 1\leq i\leq g-3.
\]
From now on, by abuse of notation, $\mathbf{L}_i$ and $\mathbf{Q}_i$ denote these generator representations, i.e. the factors $R(-i-1)$, $R(-i-2)$ are suppressed. For identities in the range $1\leq p\leq g-2$ we also set, as bookkeeping conventions for the two rows, $\mathbf{Q}_0=0$ and $\mathbf{L}_{g-2}=0$ (note that $\dim E=g-2$, so the formula for $\mathbf{L}_{g-2}$ gives $\bigwedge^{g-1}E=0$); these conventions do not affect $K_{g-2,3}(X,\Omega_X)\cong D$. Their decompositions into the irreducible characters $\chi_0,\chi_1,\chi_2$ are as follows. That these are the $kG$-module structures of $K_{i,1}(X,\Omega_X)$ and $K_{i,2}(X,\Omega_X)$ is shown in Proposition~\ref{prop:linear-strand} and Corollary~\ref{cor:trig-individual} below.}\\

\begin{proposition}\label{prop:LQ}
In the notation of Theorem \ref{Schreyer}, one has for $1\leq p \leq {g-3}$ that
\begin{align*}
\mathbf{L}_p &\cong p \cdot \bigoplus_{j=0}^{p+1} \left[ \binom{s/3-2}{j} \binom{2s/3-2}{p+1-j} \chi_{(p+j+1) \;{\rm mod}\; 3} \right]
\text{ and }\\
\mathbf{Q}_p &\cong (s-p-4) \cdot \bigoplus_{j=0}^{p-1} \left[ \binom{s/3-2}{j} \binom{2s/3-2}{p-1-j} \chi_{(p+j-1) \;{\rm mod}\; 3} \right].
\end{align*}
\end{proposition}
\begin{proof}
For $\mathcal{L}= \pi^* \cO_{\mathbb{P}^1}(1)$, the global sections $U=H^0(\mathcal{L})$ have $k$-basis $\{1,x\}$; both basis elements are invariant under the action of $G$ and so $U\cong 2\chi_0$. Further, if $V=H^0(\Omega_X)$, then
\[
E=H^0(\Omega_X\otimes\mathcal{L}^{-1})=\{\omega\in V:x\omega\in V\},
\]
and so the $k$-basis of $V$ in eq. (\ref{eq:basis}), gives rise to the $k$-basis of $E$
\[
 \left\{ \frac{x^k dx}{y} \;\middle|\; 0 \le k \le \frac{s}{3}-3 \right\} \;\cup\; \left\{ \frac{x^k dx}{y^2} \;\middle|\; 0 \le k \le \frac{2s}{3}-3 \right\}.
\]
One then directly verifies that
\[
E \cong \left(\frac{s}{3}-2\right)\chi_2 \;\oplus\; \left(\frac{2s}{3}-2\right)\chi_1
\]
 and computes, using the identities of eq. (\ref{eq:identities})
\[
\Sym_k U\cong  (k+1)\chi_0\quad\text{ and }\quad
\bigwedge^k E \cong \bigoplus_{j=0}^{k} \left[ \binom{s/3-2}{j} \binom{2s/3-2}{k-j} \chi_{(k+j) \;{\rm mod}\; 3} \right]
\]
The desired formulas follow immediately by substituting the above into the formulas of Theorem \ref{Schreyer}. {(The formula for $\mathbf{L}_p$ also gives $\mathbf{L}_{g-2}=0$.)}

{\noindent{\em The last term of the resolution.} With the natural action on the canonical ring, the generator representation of the last free module $R(-g-1)$ is $K_{g-2,3}(X,\Omega_X)\cong D=\det V$ (Proposition~\ref{prop:vanish}), so equivariantly the last module is $D\otimes R(-g-1)$. Since $V\cong\left(\frac{2s}{3}-1\right)\chi_1\oplus\left(\frac{s}{3}-1\right)\chi_2$, one has $D=\chi_{(2s/3-1)+2(s/3-1)\bmod 3}=\chi_{s/3\bmod 3}$; for instance $D=\chi_0$ for $s=9$ and $D=\chi_1$ for $s=12$.}

\end{proof}

{
\begin{proposition}\label{prop:linear-strand}
For $1\leq p\leq g-2$ there is an isomorphism of $kG$-modules
\[
K_{p,1}(X,\Omega_X)\cong \bigwedge^{p+1}E\otimes D_{p-1}(U)\otimes \bigwedge^2 U\cong p\cdot\bigwedge^{p+1}E,
\]
where $D_{p-1}(U)$ denotes the divided power. In particular $K_{p,1}(X,\Omega_X)\cong\mathbf{L}_p$ as $kG$-modules, with $\mathbf{L}_p$ as in Proposition~\ref{prop:LQ}.
\end{proposition}
\begin{proof}
Let $R={\rm Sym}\,V$ and let $I_S\subset I_C\subset R$ be the homogeneous ideals of the scroll $S$ and of the canonical model $C$ of $X$. Since $x$ is $G$-invariant, the multiplication map $U\otimes E\rightarrow V$ is $G$-equivariant, and hence so is the map of free $R$-modules
\[
\varphi: E\otimes R(-1)\rightarrow U^\vee\otimes R,\qquad e\otimes 1\mapsto (u\mapsto ue).
\]
The ideal of $2\times 2$ minors of $\varphi$ is $I_S$, and since $S$ has the expected codimension $(g-2)-2+1$, the Eagon--Northcott complex of $\varphi$ is a minimal free resolution of $R/I_S$, see \cite[Chapter 6]{MR2103875}. The Eagon--Northcott complex is functorial in $\varphi$, so it is a complex of $G$-equivariant $R$-modules, and its $p$-th term is $\bigwedge^{p+1}E\otimes D_{p-1}(U)\otimes\bigwedge^2U\otimes R(-p-1)$. By minimality,
\[
{\rm Tor}^R_p(R/I_S,k)_{p+1}\cong \bigwedge^{p+1}E\otimes D_{p-1}(U)\otimes\bigwedge^2U
\]
as $kG$-modules.

On the other hand, $C$ is projectively normal, so $R/I_C\cong\bigoplus_{q}H^0(\Omega_X^{\otimes q})$, and $K_{p,q}(X,\Omega_X)$ is the cohomology of the Koszul complex $\bigwedge^{p+1}V\otimes (R/I_C)_{q-1}\to\bigwedge^{p}V\otimes (R/I_C)_{q}\to\bigwedge^{p-1}V\otimes (R/I_C)_{q+1}$, that is, $K_{p,q}(X,\Omega_X)\cong{\rm Tor}^R_p(R/I_C,k)_{p+q}$ as $kG$-modules. Let $M=I_C/I_S$. By Theorem~\ref{Schreyer}, $(I_C)_2$ has dimension $\dim\bigwedge^2E=\dim (I_S)_2$, so $M_d=0$ for $d\leq 2$, and computing ${\rm Tor}$ with the Koszul complex of $M$ shows that ${\rm Tor}^R_j(M,k)_d=0$ for $d<j+3$. The short exact sequence $0\rightarrow M\rightarrow R/I_S\rightarrow R/I_C\rightarrow 0$ of graded $G$-equivariant $R$-modules gives the exact sequence of $kG$-modules
\[
{\rm Tor}^R_p(M,k)_{p+1}\rightarrow{\rm Tor}^R_p(R/I_S,k)_{p+1}\rightarrow{\rm Tor}^R_p(R/I_C,k)_{p+1}\rightarrow{\rm Tor}^R_{p-1}(M,k)_{p+1},
\]
whose outer terms vanish. Hence $K_{p,1}(X,\Omega_X)\cong{\rm Tor}^R_p(R/I_S,k)_{p+1}$. Finally $U\cong 2\chi_0$, so $D_{p-1}(U)\cong p\,\chi_0$ and $\bigwedge^2U\cong\chi_0$.
\end{proof}
}

{
\begin{lemma}
\label{lem:trigonal-character-comparison}
For the trigonal Kummer curve of Section~\ref{sec:Kummer-setup}, let $3\mid s$, $s\geq9$,
and $1\leq p\leq g-2=s-4$. Then
\[
\begin{aligned}
 &[K_{p,1}(X,\Omega_X)]-[K_{p-1,2}(X,\Omega_X)]\\
 &\qquad=[\mathbb{S}_p(V)]-
 \left[H^0\!\left(X,\bigwedge^{p-1}\mathcal{M}_{\Omega_X}
                         \otimes\Omega_X^{\otimes2}\right)\right]\\
 &\qquad=[\mathbf{L}_p]-[\mathbf{Q}_{p-1}].
\end{aligned}
\]
Here $\mathbf{L}_p,\mathbf{Q}_{p-1}$ denote the finite-dimensional generator
representations in Proposition~\ref{prop:LQ}; polynomial-ring factors and
grading shifts are suppressed, and $\mathbf{Q}_0=0$.
\end{lemma}

\begin{proof}
The first equality is Corollary~\ref{cor:bettivanish}. For the second, set
$a=2s/3-2$, $b=s/3-2$, and
\[
 T_{k,r}=\sum_{\substack{j\in\mathbb Z\\k+j\equiv r\; (\mathrm{mod}\,3)}}
          \binom{a}{k-j}\binom{b}{j}.
\]
Character subscripts are taken modulo $3$, and out-of-range binomial
coefficients are zero. Let $\sigma_{p,r}$ and $h_{q,r}$ denote the
coefficients of $\chi_r$ in $[\mathbb{S}_p(V)]$ and
$[H^0(X,\bigwedge^q\mathcal{M}_{\Omega_X}\otimes\Omega_X^{\otimes2})]$,
respectively.

In the formula for $h_{q,r}$, the congruence $d\equiv q-i+j$
replaces the fractional part by $\{(q+j+1-r)/3\}$.
Thus $\sum_{i=0}^{k}(-1)^i\binom{a+1}{k-i}=\binom{a}{k}$ gives
\[
 h_{q,r}=\frac13\binom{s-3}{q}(2s-2q-9)
 +s\sum_j\binom{a}{q-j}\binom{b+1}{j}
             \left\{\frac{q+j+1-r}{3}\right\}.
\]
{We evaluate the two parts separately, using $s-4=a+b$, $s/3=b+2$ and $a=2b+2$.

\emph{Global part.} From $\binom{s-3}{q}=\binom{s-4}{q}+\binom{s-4}{q-1}$ and $q\binom{s-3}{q}=(s-3)\binom{s-4}{q-1}$ one gets $(2s-2q-9)\binom{s-3}{q}=(2s-9)\binom{s-4}{q}-3\binom{s-4}{q-1}$, and $2s-9=3(2b+1)$, so
\[
 \frac13\binom{s-3}{q}(2s-2q-9)
 =(2b+1)\binom{s-4}{q}-\binom{s-4}{q-1}.
\]
By Vandermonde's identity $\binom{s-4}{k}=T_{k,0}+T_{k,1}+T_{k,2}$, so the global part contributes
$(2b+1)\left(T_{q,r}+T_{q,r-1}+T_{q,r-2}\right)-\left(T_{q-1,r}+T_{q-1,r-1}+T_{q-1,r-2}\right)$
to $h_{q,r}$, for every $r$.

\emph{Local part.} Let $c\in\{0,1,2\}$ be the residue of $q+j-r$ modulo $3$. Then
\[
\left\{\frac{q+j+1-r}{3}\right\}=\left\{\frac{c+1}{3}\right\}=
\begin{cases}
1/3, & c=0,\\
2/3, & c=1,\\
0, & c=2.
\end{cases}
\]
Moreover, Pascal's identity $\binom{b+1}{j}=\binom{b}{j}+\binom{b}{j-1}$ and the substitution $j\mapsto j+1$ in the second term give, for every $\rho$,
\[
\sum_{\substack{j\in\mathbb Z\\ q+j\equiv\rho\;(\mathrm{mod}\,3)}}\binom{a}{q-j}\binom{b+1}{j}=T_{q,\rho}+T_{q-1,\rho-2}.
\]
Applying this with $\rho=r$ (the case $c=0$) and $\rho=r+1\equiv r-2$ (the case $c=1$), the local part of $h_{q,r}$ equals
\[
s\left(\frac13\left(T_{q,r}+T_{q-1,r-2}\right)+\frac23\left(T_{q,r-2}+T_{q-1,r-1}\right)\right)
=(b+2)\left(T_{q,r}+T_{q-1,r-2}+2T_{q,r-2}+2T_{q-1,r-1}\right).
\]

\emph{Collecting terms.} Adding the two contributions, the coefficients are
$(2b+1)+(b+2)=s-3$ for $T_{q,r}$; $2b+1$ for $T_{q,r-1}$; $(2b+1)+2(b+2)=2a+1$ for $T_{q,r-2}$; $-1$ for $T_{q-1,r}$; $-1+2(b+2)=a+1$ for $T_{q-1,r-1}$; and $-1+(b+2)=b+1$ for $T_{q-1,r-2}$. Hence}
\begin{equation}\label{eq:trigonal-comparison-h}
\begin{aligned}
 h_{q,r}={}&(s-3)T_{q,r}+(2b+1)T_{q,r-1}+(2a+1)T_{q,r-2}\\
          &-T_{q-1,r}+(a+1)T_{q-1,r-1}+(b+1)T_{q-1,r-2}.
\end{aligned}
\end{equation}
Applying
$(k-1)\binom{a+2}{k}=(a+1)\binom{a+1}{k-1}-\binom{a+1}{k}$
with $k=p+1-j$ to the formula for $\sigma_{p,r}$, followed by
Pascal's identity and subtraction of
\eqref{eq:trigonal-comparison-h}, gives
\begin{equation}\label{eq:trigonal-comparison-difference}
\begin{aligned}
 \sigma_{p,r}-h_{p-1,r}
 ={}&-T_{p+1,r}+T_{p-2,r}\\
    &+a(T_{p,r-1}-T_{p-1,r-2})
     +b(T_{p,r-2}-T_{p-1,r-1}).
\end{aligned}
\end{equation}
Finally, Pascal's identity and
$n\binom{n-1}{k-1}=k\binom{n}{k}$,
$n\binom{n-1}{k}=(n-k)\binom{n}{k}$ yield
\[
\begin{aligned}
 &a(T_{p,r-1}-T_{p-1,r-2})+b(T_{p,r-2}-T_{p-1,r-1})\\
 &\quad=\sum_{p+j+1\equiv r}
 \left[a\binom{a-1}{p-j}\binom{b}{j}
       +b\binom{a}{p+1-j}\binom{b-1}{j-1}\right]\\
 &\qquad-\sum_{p+j-2\equiv r}
 \left[a\binom{a-1}{p-2-j}\binom{b}{j}
       +b\binom{a}{p-2-j}\binom{b-1}{j}\right]\\
 &\quad=(p+1)T_{p+1,r}-(s-p-2)T_{p-2,r},
\end{aligned}
\]
where both sums range over $j\in\mathbb Z$ and $a+b=s-4$.
Substitution into \eqref{eq:trigonal-comparison-difference} gives
\[
 \sigma_{p,r}-h_{p-1,r}
 =pT_{p+1,r}-(s-p-3)T_{p-2,r},
\]
which is the coefficient of $\chi_r$ in $[\mathbf{L}_p]-[\mathbf{Q}_{p-1}]$
from Proposition~\ref{prop:LQ}.
\end{proof}

\begin{corollary}\label{cor:trig-individual}
For $1\leq p\leq g-2$ there are isomorphisms of $kG$-modules $K_{p,1}(X,\Omega_X)\cong\mathbf{L}_p$ and $K_{p-1,2}(X,\Omega_X)\cong\mathbf{Q}_{p-1}$.
\end{corollary}
\begin{proof}
The first isomorphism is Proposition~\ref{prop:linear-strand}. Lemma~\ref{lem:trigonal-character-comparison} then gives $[K_{p-1,2}(X,\Omega_X)]=[\mathbf{Q}_{p-1}]$ in ${\rm Rep}_k(G)$, and $kG$-modules with the same class are isomorphic.
\end{proof}

\noindent\textit{Remark.}
The calculation in Section~\ref{sec:Kummer-explicit} uses only Corollary~\ref{cor:bettivanish},
equivariant Riemann--Roch and the character formula for $\mathbb{S}_p(V)$.
Lemma~\ref{lem:trigonal-character-comparison} establishes its agreement
with the virtual differences obtained from Schreyer's classical
resolution. The role of that resolution in this comparison is to
provide an independent verification of our formulas. No novelty is
claimed for the resolution or its mapping-cone construction. The
separate linear and quadratic representations in Tables~\ref{tab:betti9} and~\ref{tab:betti12} use
the additional information supplied by the scroll (Proposition~\ref{prop:linear-strand} and Corollary~\ref{cor:trig-individual}); they are not,
in general, recoverable from the virtual differences alone.

}

Putting $s=9$, one obtains the {linear and quadratic rows of the} equivariant Betti table of the trigonal Kummer curve of genus $g=7$ {(the omitted endpoints are $K_{0,0}=\chi_0$ and $K_{5,3}=D=\chi_0$)}

\begin{table}[h]
\centering
\renewcommand{\arraystretch}{1.5}
\begin{tabular}{|l|c|c|c|c|}
\hline
& $\mathbf{p=1}$ & $\mathbf{p=2}$ & $\mathbf{p=3}$ & $\mathbf{p=4}$ \\ \hline
\textbf{Linear}
& $4\chi_0 \oplus 6\chi_2$
& $8\chi_0 \oplus 12\chi_1$
& $3\chi_1 \oplus 12\chi_2$
& $4\chi_0$ \\ \hline
\textbf{Quadratic}
& $4\chi_0$
& $12\chi_1 \oplus 3\chi_2$
& $8\chi_0 \oplus 12\chi_2$
& $4\chi_0 \oplus 6\chi_1$ \\ \hline
\textbf{Total}
& $8\chi_0 \oplus 6\chi_2$
& $8\chi_0 \oplus 24\chi_1 \oplus 3\chi_2$
& $8\chi_0 \oplus 3\chi_1 \oplus 24\chi_2$
& $8\chi_0 \oplus 6\chi_1$ \\ \hline
\textbf{Rank} & 14 & 35 & 35 & 14 \\ \hline
\end{tabular}
\caption{Equivariant Betti Table {(linear and quadratic rows)} for the genus 7 trigonal Kummer curve ($s=9$).}\label{tab:betti9}
\end{table}

For $s=12$, the {linear and quadratic rows of the} equivariant Betti table of the genus 10 trigonal Kummer curve {are (the omitted endpoints are $K_{0,0}=\chi_0$ and $K_{8,3}=D=\chi_1$)}

\begin{table}[h]
\centering
\renewcommand{\arraystretch}{1.5}
\resizebox{\textwidth}{!}{
\begin{tabular}{|l|c|c|c|c|c|c|c|}
\hline
& $\mathbf{p=1}$ & $\mathbf{p=2}$ & $\mathbf{p=3}$ & $\mathbf{p=4}$ & $\mathbf{p=5}$ & $\mathbf{p=6}$ & $\mathbf{p=7}$ \\ \hline
\textbf{Linear}
& $12\chi_0 \oplus 1\chi_1 \oplus 15\chi_2$
& $40\chi_0 \oplus 60\chi_1 \oplus 12\chi_2$
& $45\chi_0 \oplus 45\chi_1 \oplus 120\chi_2$
& $120\chi_0 \oplus 80\chi_1 \oplus 24\chi_2$
& $5\chi_0 \oplus 60\chi_1 \oplus 75\chi_2$
& $36\chi_0 \oplus 12\chi_2$
& $7\chi_1$ \\ \hline
\textbf{Quadratic}
& $7\chi_0$
& $36\chi_1 \oplus 12\chi_2$
& $60\chi_0 \oplus 5\chi_1 \oplus 75\chi_2$
& $80\chi_0 \oplus 120\chi_1 \oplus 24\chi_2$
& $45\chi_0 \oplus 45\chi_1 \oplus 120\chi_2$
& $60\chi_0 \oplus 40\chi_1 \oplus 12\chi_2$
& $1\chi_0 \oplus 12\chi_1 \oplus 15\chi_2$ \\ \hline
\textbf{Total}
& $19\chi_0 \oplus 1\chi_1 \oplus 15\chi_2$
& $40\chi_0 \oplus 96\chi_1 \oplus 24\chi_2$
& $105\chi_0 \oplus 50\chi_1 \oplus 195\chi_2$
& $200\chi_0 \oplus 200\chi_1 \oplus 48\chi_2$
& $50\chi_0 \oplus 105\chi_1 \oplus 195\chi_2$
& $96\chi_0 \oplus 40\chi_1 \oplus 24\chi_2$
& $1\chi_0 \oplus 19\chi_1 \oplus 15\chi_2$ \\ \hline
\textbf{Rank} & 35 & 160 & 350 & 448 & 350 & 160 & 35 \\ \hline
\end{tabular}
}
\caption{Equivariant Betti Table {(linear and quadratic rows)} for the genus 10 Kummer curve ($s=12$).}\label{tab:betti12}
\end{table}

Considering the ``antidiagonal'' differences $[\mathbf{L}_p]-[\mathbf{Q}_{p-1}]$ in each case as virtual representations, one retrieves the exact same values as in Table \ref{tab:s=9} and Table \ref{tab:s=12} respectively.

\begin{table}[h]
\centering
\renewcommand{\arraystretch}{1.5}
\begin{minipage}[b]{0.45\textwidth}
\centering
\begin{tabular}{|c|c|}
\hline
$\mathbf{p}$ & $[\mathbf{L}_p]-[\mathbf{Q}_{p-1}]$ \\ \hline
$\mathbf{1}$ & $4\chi_0 + 6\chi_2$ \\ \hline
$\mathbf{2}$ & $4\chi_0 + 12\chi_1$ \\ \hline
$\mathbf{3}$ & $-9\chi_1 + 9\chi_2$ \\ \hline
$\mathbf{4}$ & $-4\chi_0 - 12\chi_2$ \\ \hline
$\mathbf{5}$ & $-4\chi_0 - 6\chi_1$ \\ \hline
\end{tabular}
\caption{$s=9,\;g=7$.}
\end{minipage}
\hspace{0.5cm}
\begin{minipage}[b]{0.45\textwidth}
\centering
\begin{tabular}{|c|c|}
\hline
$\mathbf{p}$ & $[\mathbf{L}_p]-[\mathbf{Q}_{p-1}]$ \\ \hline
$\mathbf{1}$ & $12\chi_0 + 1\chi_1 + 15\chi_2$ \\ \hline
$\mathbf{2}$ & $33\chi_0 + 60\chi_1 + 12\chi_2$ \\ \hline
$\mathbf{3}$ & $45\chi_0 + 9\chi_1 + 108\chi_2$ \\ \hline
$\mathbf{4}$ & $60\chi_0 + 75\chi_1 - 51\chi_2$ \\ \hline
$\mathbf{5}$ & $-75\chi_0 - 60\chi_1 + 51\chi_2$ \\ \hline
$\mathbf{6}$ & $-9\chi_0 - 45\chi_1 - 108\chi_2$ \\ \hline
$\mathbf{7}$ & $-60\chi_0 - 33\chi_1 - 12\chi_2$ \\ \hline
$\mathbf{8}$ & $-1\chi_0 - 12\chi_1 - 15\chi_2$ \\ \hline
\end{tabular}
\caption{$s=12,\;g=10$.}
\end{minipage}
\end{table}

{Both tables agree with the closed formula which, by Lemma~\ref{lem:trigonal-character-comparison} and Proposition~\ref{prop:LQ}, holds for every $s$ divisible by $3$ with $s\geq 9$ and every $1\leq p\leq g-2$:
\begin{equation*}
[K_{p,1}(X,\Omega_X)]-[K_{p-1,2}(X,\Omega_X)] = \sum_{j=0}^{p+1} \binom{s/3-2}{j} \left( p \binom{2s/3-2}{p+1-j} - (s-p-3) \binom{2s/3-2}{p-2-j} \right) \chi_{(p+j+1) \;{\rm mod}\; 3}.
\end{equation*}
}

{
\section{Concluding remarks}
For a canonical curve $X$ with an action of a finite group $G$, Corollary~\ref{cor:bettivanish}, together with Theorems~\ref{th:mainwedge}, \ref{th:genchars} and~\ref{th:Molien}, expresses the virtual representations $[K_{p,1}(X,\Omega_X)]-[K_{p-1,2}(X,\Omega_X)]$ in terms of the genus of $X/G$ and the ramification data of $X\to X/G$ alone, without computing a resolution. By Remark~\ref{rem:individual}, {under the vanishing hypothesis made there (Green's conjecture for $X$),} outside the middle range $c<p<g-1-c$ these differences are the classes of individual Koszul cohomology groups. Inside the middle range further input is needed; for the trigonal Kummer curves of Section~\ref{sec:Kummer} it is supplied by Schreyer's resolution (Theorem~\ref{Schreyer}), and Lemma~\ref{lem:trigonal-character-comparison} shows that the two approaches agree. An equivariant version of the duality theorem for Koszul cohomology, which refines Remark~\ref{rem:individual}, will be treated elsewhere.
}

\bibliographystyle{plain}
 \bibliography{KKGeneral.bib}

\def\cprime{$'$}
\begin{thebibliography}{10}

\bibitem{MR2573635}
Marian Aprodu and Jan Nagel.
\newblock {\em Koszul cohomology and algebraic geometry}, volume~52 of {\em
  University Lecture Series}.
\newblock American Mathematical Society, Providence, RI, 2010.

\bibitem{BensonPolInv}
D.~J. Benson.
\newblock {\em Polynomial invariants of finite groups}, volume 190 of {\em
  London Mathematical Society Lecture Note Series}.
\newblock Cambridge University Press, Cambridge, 1993.

\bibitem{MR1038525}
Charles~W. Curtis and Irving Reiner.
\newblock {\em Methods of representation theory. {V}ol. {I}}.
\newblock Wiley Classics Library. John Wiley \& Sons, Inc., New York, 1990.
\newblock With applications to finite groups and orders, Reprint of the 1981
  original, A Wiley-Interscience Publication.

\bibitem{Eisenbud:95}
David Eisenbud.
\newblock {\em Commutative algebra}.
\newblock Springer-Verlag, New York, 1995.
\newblock With a view toward algebraic geometry.

\bibitem{MR2103875}
David Eisenbud.
\newblock {\em The geometry of syzygies}, volume 229 of {\em Graduate Texts in
  Mathematics}.
\newblock Springer-Verlag, New York, 2005.
\newblock A second course in commutative algebra and algebraic geometry.

\bibitem{MR589254}
G.~Ellingsrud and K.~L\o~nsted.
\newblock An equivariant {L}efschetz formula for finite reductive groups.
\newblock {\em Math. Ann.}, 251(3):253--261, 1980.

\bibitem{MR3729076}
Gavril Farkas.
\newblock Progress on syzygies of algebraic curves.
\newblock In {\em Moduli of curves}, volume~21 of {\em Lect. Notes Unione Mat.
  Ital.}, pages 107--138. Springer, Cham, 2017.

\bibitem{MR1644323}
William Fulton.
\newblock {\em Intersection theory}, volume~2 of {\em Ergebnisse der Mathematik
  und ihrer Grenzgebiete. 3. Folge. A Series of Modern Surveys in Mathematics
  [Results in Mathematics and Related Areas. 3rd Series. A Series of Modern
  Surveys in Mathematics]}.
\newblock Springer-Verlag, Berlin, second edition, 1998.

\bibitem{MR1153249}
William Fulton and Joe Harris.
\newblock {\em Representation theory}, volume 129 of {\em Graduate Texts in
  Mathematics}.
\newblock Springer-Verlag, New York, 1991.
\newblock A first course, Readings in Mathematics.

\bibitem{Galetto2022FiniteGroupCharacters}
Federico Galetto.
\newblock Finite group characters on free resolutions.
\newblock {\em J. Symbolic Comput.}, 113:29--38, 2022.

\bibitem{Galetto2023BettiCharacters}
Federico Galetto.
\newblock Setting the scene for {B}etti characters.
\newblock {\em J. Softw. Algebra Geom.}, 13(1):45--51, 2023.

\bibitem{MR4225278}
Ulrich G\"{o}rtz and Torsten Wedhorn.
\newblock {\em Algebraic geometry {I}. {S}chemes---with examples and
  exercises}.
\newblock Springer Studium Mathematik---Master. Springer Spektrum, Wiesbaden,
  [2020] \copyright 2020.
\newblock Second edition [of 2675155].

\bibitem{MR739785}
Mark~L. Green.
\newblock Koszul cohomology and the geometry of projective varieties.
\newblock {\em J. Differential Geom.}, 19(1):125--171, 1984.
\newblock With an appendix by Robert Lazarsfeld and Green.

\bibitem{MR1082354}
Mark~L. Green.
\newblock Koszul cohomology and geometry.
\newblock In {\em Lectures on {R}iemann surfaces ({T}rieste, 1987)}, pages
  177--200. World Sci. Publ., Teaneck, NJ, 1989.

\bibitem{Hartshorne:77}
Robin Hartshorne.
\newblock {\em Algebraic geometry}.
\newblock Springer-Verlag, New York, 1977.
\newblock Graduate Texts in Mathematics, No. 52.

\bibitem{MR1510634}
David Hilbert.
\newblock Ueber die {T}heorie der algebraischen {F}ormen.
\newblock {\em Math. Ann.}, 36(4):473--534, 1890.

\bibitem{MR1335917}
Friedrich Hirzebruch.
\newblock {\em Topological methods in algebraic geometry}.
\newblock Classics in Mathematics. Springer-Verlag, Berlin, 1978 edition, 1995.
\newblock Translated from the German and Appendix One by R. L. E.
  Schwarzenberger, With a preface to the third English edition by the author
  and Schwarzenberger, Appendix Two by A. Borel.

\bibitem{MR2665168}
Daniel Huybrechts and Manfred Lehn.
\newblock {\em The geometry of moduli spaces of sheaves}.
\newblock Cambridge Mathematical Library. Cambridge University Press,
  Cambridge, second edition, 2010.

\bibitem{KaranProc}
Sotiris Karanikolopoulos and Aristides Kontogeorgis.
\newblock Integral representations of cyclic groups acting on relative
  holomorphic differentials of deformations of curves with automorphisms.
\newblock {\em Proc. Amer. Math. Soc.}, 142(7):2369--2383, 2014.

\bibitem{MR2205171}
Bernhard K\"{o}ck.
\newblock Computing the equivariant {E}uler characteristic of {Z}ariski and
  \'{e}tale sheaves on curves.
\newblock {\em Homology Homotopy Appl.}, 7(3):83--98, 2005.

\bibitem{MR42428}
J.~L. Koszul.
\newblock Sur un type d'alg\`ebres diff\'{e}rentielles en rapport avec la
  transgression.
\newblock In {\em Colloque de topologie (espaces fibr\'{e}s), {B}ruxelles,
  1950}, pages 73--81. Georges Thone, Li\`ege, 1951.

\bibitem{MR1082360}
Robert Lazarsfeld.
\newblock A sampling of vector bundle techniques in the study of linear series.
\newblock In {\em Lectures on {R}iemann surfaces ({T}rieste, 1987)}, pages
  500--559. World Sci. Publ., Teaneck, NJ, 1989.

\bibitem{MR1011461}
Hideyuki Matsumura.
\newblock {\em Commutative ring theory}, volume~8 of {\em Cambridge Studies in
  Advanced Mathematics}.
\newblock Cambridge University Press, Cambridge, second edition, 1989.
\newblock Translated from the Japanese by M. Reid.

\bibitem{MR1304906}
D.~Mumford, J.~Fogarty, and F.~Kirwan.
\newblock {\em Geometric invariant theory}, volume~34 of {\em Ergebnisse der
  Mathematik und ihrer Grenzgebiete (2) [Results in Mathematics and Related
  Areas (2)]}.
\newblock Springer-Verlag, Berlin, third edition, 1994.

\bibitem{Saint-Donat73}
B.~Saint-Donat.
\newblock On {P}etri's analysis of the linear system of quadrics through a
  canonical curve.
\newblock {\em Math. Ann.}, 206:157--175, 1973.

\bibitem{MR849058}
Frank-Olaf Schreyer.
\newblock Syzygies of canonical curves and special linear series.
\newblock {\em Math. Ann.}, 275(1):105--137, 1986.

\bibitem{MR3979084}
Chao Sun.
\newblock A note on equivariantization of additive categories and triangulated
  categories.
\newblock {\em J. Algebra}, 534:483--530, 2019.

\bibitem{MR921490}
R.~W. Thomason.
\newblock Algebraic {$K$}-theory of group scheme actions.
\newblock In {\em Algebraic topology and algebraic {$K$}-theory ({P}rinceton,
  {N}.{J}., 1983)}, volume 113 of {\em Ann. of Math. Stud.}, pages 539--563.
  Princeton Univ. Press, Princeton, NJ, 1987.

\bibitem{MR1941089}
Claire Voisin.
\newblock Green's generic syzygy conjecture for curves of even genus lying on a
  {$K3$} surface.
\newblock {\em J. Eur. Math. Soc. (JEMS)}, 4(4):363--404, 2002.

\bibitem{MR2157134}
Claire Voisin.
\newblock Green's canonical syzygy conjecture for generic curves of odd genus.
\newblock {\em Compos. Math.}, 141(5):1163--1190, 2005.

\bibitem{Webb2016}
Peter Webb.
\newblock {\em A Course in Finite Group Representation Theory}, volume 161 of
  {\em Cambridge Studies in Advanced Mathematics}.
\newblock Cambridge University Press, Cambridge, 2016.

\end{thebibliography}

\end{document}